\documentclass[letterpaper]{article}

\usepackage[letterpaper,left=1.5in,right=1.5in,top=1in,bottom=1in]{geometry}

\newcommand{\myauthor}{Benjamin Antieau, Tobias Barthel, and David Gepner}
\newcommand{\mytitle}{On localization sequences in the algebraic $K$-theory of ring spectra}
\newcommand{\pdftitle}{\mytitle}

\author{Benjamin Antieau\footnote{Benjamin Antieau was supported by NSF Grant DMS-1461847.},
Tobias Barthel, and David Gepner\footnote{David Gepner was supported by NSF Grant
DMS-1406529.}}
\title{On localization sequences in the algebraic $K$-theory of ring spectra}

\usepackage[pdfstartview=FitH,
            pdfauthor={\myauthor},
            pdftitle={\pdftitle},
            colorlinks,
            linkcolor=reference,
            citecolor=citation,
            urlcolor=e-mail,
            ]{hyperref}

\usepackage{mathrsfs}
\usepackage[mathscr]{euscript}
\usepackage{amsmath}
\usepackage{amscd}
\usepackage{amsbsy}
\usepackage{amssymb}
\usepackage{microtype}
\usepackage{enumerate}
\usepackage[all,cmtip]{xy}
\usepackage{tikz}
\usetikzlibrary{matrix,arrows}
\usepackage{dsfont}
\usepackage[abbrev,lite]{amsrefs}
\usepackage{amsthm}
\usepackage{comment}

\usepackage{color}
\definecolor{todo}{rgb}{1,0,0}
\definecolor{conditional}{rgb}{0,1,0}
\definecolor{e-mail}{rgb}{0,.40,.80}
\definecolor{reference}{rgb}{.20,.60,.22}
\definecolor{mrnumber}{rgb}{.80,.40,0}
\definecolor{citation}{rgb}{0,.40,.80}

\DeclareMathOperator{\Ho}{Ho}

\newcommand{\perf}{\mathrm{perf}}

\DeclareMathOperator{\cofib}{cofib}

\DeclareMathOperator{\id}{id}

\DeclareMathOperator{\op}{op}

\DeclareMathOperator{\im}{im}

\DeclareMathOperator{\eq}{eq}

\DeclareMathOperator*{\colim}{colim}

\newcommand{\rwe}{\tilde{\rightarrow}}

\newcommand{\we}{\simeq}
\newcommand{\iso}{\cong}

\newcommand{\KU}{\mathrm{KU}}
\newcommand{\ku}{\mathrm{ku}}
\renewcommand{\epsilon}{\varepsilon}

\newcommand{\Tr}{\mathrm{Tr}}

\newcommand{\ex}{\mathrm{ex}}

\newcommand{\qc}{\mathrm{qc}}

\newcommand{\Bcyc}{\mathrm{B}^{\mathrm{cyc}}}
\DeclareMathOperator{\THH}{THH}
\DeclareMathOperator{\HH}{HH}
\DeclareMathOperator{\K}{\mathrm{K}}

\newcommand{\Kloc}{\mathds{K}}
\newcommand{\BP}[1]{\mathrm{BP}{\langle #1 \rangle}}
\newcommand{\A}[1]{\mathrm{A}{\langle #1 \rangle}}

\newcommand{\BPP}{\mathrm{BP}}
\newcommand{\MU}{\mathrm{MU}}

\newcommand{\Sp}{\mathrm{Sp}}

\newcommand{\Mod}{\mathrm{Mod}}

\newcommand{\Alg}{\mathrm{Alg}}
\newcommand{\CAlg}{\mathrm{CAlg}}

\newcommand{\Cat}{\mathrm{Cat}}
\DeclareMathOperator{\Spec}{Spec}

\DeclareMathOperator{\GL}{GL}

\DeclareMathOperator{\BGL}{BGL}

\DeclareMathOperator{\Sym}{Sym}

\DeclareMathOperator{\Hoh}{H}
\DeclareMathOperator{\Eoh}{E}

\DeclareMathOperator{\Tor}{Tor}
\DeclareMathOperator{\Ext}{Ext}

\newcommand{\Hom}{\mathrm{Hom}}
\newcommand{\Map}{\mathrm{Map}} % enriched maps
\newcommand{\Fun}{\mathrm{Fun}}
\newcommand{\map}{\mathrm{map}} % mapping spaces (values in S)

\DeclareMathOperator{\End}{End}
\DeclareMathOperator{\Aut}{\!Aut}

\newcommand{\Lrm}{\mathrm{L}}

\newcommand{\Mrm}{\mathrm{M}}

\newcommand{\Hrm}{\mathrm{H}\,}

\newcommand{\Brm}{\mathrm{B}}

\newcommand{\Drm}{\mathrm{D}}

\newcommand{\Erm}{\mathrm{E}}

\newcommand{\Cscr}{\mathscr{C}} % general R-linear categories
\newcommand{\Dscr}{\mathscr{D}}
\newcommand{\EEcr}{\mathscr{E}}

\newcommand{\CC}{\mathds{C}}

\newcommand{\QQ}{\mathds{Q}}
\newcommand{\ZZ}{\mathds{Z}}

\newcommand{\EE}{\mathds{E}}
\newcommand{\FF}{\mathds{F}}
\newcommand{\PP}{\mathds{P}}

\renewcommand{\SS}{\mathds{S}}
\newcommand{\WW}{\mathds{W}}

\theoremstyle{plain}
\newtheorem{theorem}{Theorem}[section]
\newtheorem{lemma}[theorem]{Lemma}
\newtheorem{proposition}[theorem]{Proposition}

\newtheorem{corollary}[theorem]{Corollary}

\theoremstyle{definition}
\newtheorem{definition}[theorem]{Definition}

\newtheorem{question}[theorem]{Question}

\newtheorem{remark}[theorem]{Remark}

\let\oldmarginpar\marginpar
\renewcommand\marginpar[1]{\-\oldmarginpar[\raggedleft\footnotesize #1]%
{\raggedright\footnotesize #1}}

\usepackage{txfonts}

\begin{document}
\maketitle

\begin{abstract}
    \noindent
    We identify the $K$-theoretic fiber of a localization of ring spectra in terms of the
    $K$-theory of the endomorphism algebra spectrum of a Koszul-type complex.
    Using this identification, we provide a negative answer to
    a question of Rognes for $n>1$ by comparing the traces of the fiber of
    the map $\K(\BP{n})\rightarrow\K(\Erm(n))$ and of $\K(\BP{n-1})$ in rational topological Hochschild
    homology.


    \paragraph{Mathematics Subject Classification 2010.}
    Primary:
    \href{http://www.ams.org/mathscinet/msc/msc2010.html?t=19Dxx&btn=Current}{19D55},
    \href{http://www.ams.org/mathscinet/msc/msc2010.html?t=55Pxx&btn=Current}{55P43}.
    Secondary:
    \href{http://www.ams.org/mathscinet/msc/msc2010.html?t=16Exx&btn=Current}{16E40},
    \href{http://www.ams.org/mathscinet/msc/msc2010.html?t=18Exx&btn=Current}{18E30},
    \href{http://www.ams.org/mathscinet/msc/msc2010.html?t=19Dxx&btn=Current}{19D10}.
\end{abstract}

\setcounter{tocdepth}{1}
\tableofcontents

\section{Introduction}\label{sec:intro}

This paper is about the algebraic $K$-theory of structured ring spectra, or $\EE_1$-rings,
occurring in chromatic homotopy theory and the Ausoni--Rognes program for computing the $K$-theory of the sphere
spectrum. The two ring spectra of interest, the truncated Brown--Peterson
spectrum $\BP{n}$ and the Johnson--Wilson theory $\Erm(n)$
($n\geq 0$), are constructed using the complex cobordism spectrum $\MU$ and exist for any prime $p$; their homotopy rings are
$$\pi_*\BP{n}\iso\ZZ_{(p)}[v_1,\ldots,v_n]
\hspace{1em}\text{and}\hspace{1em}\pi_*\Erm(n)=\ZZ_{(p)}[v_1,\ldots,v_{n-1},v_n^{\pm 1}],$$
where $v_i$ has degree $2p^i-2$.

%The following question of Rognes first appears in~\cite{ausoni-rognes}*{0.1}, and is
%sometimes referred to in the literature as either the Rognes
%conjecture~\cite{blumberg-mandell} or the Ausoni--Rognes
%conjecture~\cite{barwick-q}*{Example~5.15}.
The following well-known question\footnote{The authors
of~\cite{ausoni-rognes} refer to this statement as an expectation.
It has come to be known in the literature as a conjecture, especially in the work
of Barwick and Blumberg--Mandell.} of Rognes first appears
in~\cite{ausoni-rognes}*{0.1}; see also~\cite{barwick-q}*{Example~5.15},
\cite{barwick-rognes}*{Example~5.15},
\cite{barwick-highercats}*{Example~11.17}, and the introduction
to~\cite{blumberg-mandell}.

\begin{question}[Rognes]\label{conj:ar}
    Is the sequence
    \begin{equation*}
        \K(\BP{n-1}_p)\rightarrow\K(\BP{n}_p)\rightarrow\K(\Erm(n)_p)
    \end{equation*}
    of connective algebraic $K$-theory spectra a fiber sequence of connective
    spectra?
\end{question}

The map $\K(\BP{n}_p)\rightarrow\K(\Erm(n)_p)$ is induced by
the map $\BP{n}\rightarrow\Erm(n)$ which inverts $v_n$ and $p$-completion, while
$\K(\BP{n-1}_p)\rightarrow\K(\BP{n}_p)$ is the \emph{transfer} map, obtained by viewing
$\BP{n-1}_p$, the cofiber of multiplication by $v_n$, as a compact $\BP{n}_p$-module.

When $n=0$, this is a special case of a theorem of Quillen~\cite{quillen}*{Theorem 5}, saying that there is a fiber
sequence $\K(\FF_p)\rightarrow\K(\ZZ_p)\rightarrow\K(\QQ_p)$. In this case it is common to
let $v_0=p$ in $\ZZ_p$ and $\BP{-1}=\Hrm\FF_p$. When $n=1$, the sequence
was conjectured by Rognes and proved by Blumberg and Mandell in~\cite{blumberg-mandell}. In fact, both Quillen
and Blumberg--Mandell prove $p$-local and integral versions of these statements.

The backdrop of the question of Rognes is the Ausoni--Rognes program to compute the
algebraic $K$-theory of the sphere spectrum while keeping control of chromatic
phenomenon. The layers in the chromatic tower are closely related to
$\K(\Erm_n)$, and it is expected that $\K(\Erm_n)$ should behave in a similar
way to $\K(\Erm(n))$. However, there is a fundamental problem with computing $\K(\Erm_n)$ and $\K(\Erm(n))$: they are nonconnective ring
spectra. There are no general methods for computing the $K$-groups of nonconnective ring
spectra, and it is in general even difficult to produce candidate elements. The little
success that has been had here is to relate the $K$-theory of the nonconnective ring
spectrum to the $K$-theory of connective ring spectra, where there are a variety of methods
of computation, such as using determinants and traces or studying $\mathrm{BGL}(R)^+$.

One reason to study $\K(\BP{n}_p)$ in its own right is that it is expected to exhibit redshift, a phenomenon
visible for small values of $n$ in which the $K$-theory of a ring spectrum related to chromatic height
$n$ carries chromatic height $n+1$ information~\cite{ausoni-rognes}*{p. 7}. For example,
$v_0$ acts trivially on $\FF_p$,
whereas $\K(\FF_p)_p\we\Hrm\ZZ_p$ and hence $\K(\FF_p)$ carries a highly non-trivial
$v_0$-self map: multiplication by $p$.

We give a negative answer to Rognes' question in all of the remaining cases.

\begin{theorem}\label{thm:intromain}
    For $n>1$, the sequence
    \begin{equation*}
         \K(\BP{n-1}_p)\rightarrow\K(\BP{n}_p)\rightarrow\K(\Erm(n)_p)
    \end{equation*}
    is not a fiber sequence of connective spectra.
\end{theorem}

When $n=0,1$,
both Quillen and Blumberg--Mandell relate the fiber of $\K(\BP{n}_p)\rightarrow\K(\Erm(n)_p)$ to
$\K(\BP{n-1}_p)$ using a d\'evissage argument, and this perspective has been used by Barwick
and Lawson to prove some other examples of this localization behavior~\cite{barwick-lawson}.
Most attempts to answer Rognes' question for
$n>1$, such as the approach outlined by Barwick~\cite{barwick-q}, focus on conjectural d\'evissage
arguments. Our theorem shows that in some sense these cannot work in general. As a corollary of Theorem~\ref{thm:intromain}, the
$\infty$-category $Z(f_\star)$ of~\cite{barwick-q}*{Example 5.15} is not weakly contractible.

Our approach is Morita-theoretic. As motivation, consider the major result of
Thomason--Trobaugh~\cite{thomason-trobaugh}*{Theorem~7.4}, the localization theorem. It states that if $X$
is a quasi-compact and quasi-separated scheme, and if $U\subseteq X$ is a quasi-compact
Zariski open with complement $Z$, then there is a fiber sequence of nonconnective algebraic
$K$-theory spectra
\begin{equation*}
    \Kloc(\text{$X$ on $Z$})\rightarrow\Kloc(X)\rightarrow\Kloc(U),
\end{equation*}
where $\Kloc(\text{$X$ on $Z$})$ is the $K$-theory of perfect complexes on $X$
that are acyclic on $U$.

In general $\Kloc(\text{$X$ on $Z$})$ is \emph{not} equivalent to $\Kloc(Z)$, the $K$-theory of the closed
subscheme $Z$. The main examples when $\Kloc(\text{$X$ on $Z$})$ \emph{is} equivalent
to $\Kloc(Z)$ occur when $X$ and $Z$ are regular and noetherian, and the
argument passes through $G$-theory via d\'evissage.
However, B\"okstedt and Neeman~\cite{bokstedt-neeman}*{Proposition~6.1} showed that nevertheless $\Kloc(\text{$X$ on $Z$})$ is the
$K$-theory of a ring spectrum. More specifically, they showed that at the level of derived
categories, the kernel of the localization
\begin{equation*}
    \Drm_{\qc}(X)\rightarrow\Drm_{\qc}(U)
\end{equation*}
is generated by a single compact object $K$.
If $A=\End_X(K)^{\op}$ denotes the opposite of the dg-algebra of endomorphisms of $K$, then
\begin{equation*}
    \Kloc(A)\we\Kloc(\text{$X$ on $Z$}),
\end{equation*}
so we can write our localization fiber sequence as
\begin{equation}~\label{eq:ttring}
    \Kloc(A)\rightarrow\Kloc(X)\rightarrow\Kloc(U).
\end{equation}
Antieau and Gepner proved the analogue of this result for localizations of spectral
schemes in~\cite{ag}*{Proposition~6.9}, which motivated our approach here. However, the truncated
Brown--Peterson spectra are not known to admit $\EE_\infty$-ring structures, a
necessary input in~\cite{ag}. Noncommutative localization sequences have been studied
extensively in the $K$-theory of ordinary rings (see Neeman and Ranicki~\cite{neeman-ranicki} and the references
there). We prove a spectral noncommutative analogue
of~\eqref{eq:ttring}, which will be strong enough for our application to
Rognes' question,
and gives a partial generalization of Neeman-Ranicki. To state it, let $R$ be
an $\EE_1$-ring and let $r\in\pi_*R$ be a homogeneous element such that
$\{1,r,r^2,\ldots\}$ satisfies the right Ore condition. By Proposition
\ref{prop:ore}, there is an $\EE_1$-ring $R[r^{-1}]$ and an $\EE_1$-ring map
$R \rightarrow R[r^{-1}]$ inducing an isomorphism
$(\pi_*R)[r^{-1}]\iso\pi_*(R[r^{-1}])$.

\begin{theorem}\label{thm:abstractfiberseq}
    For $R$ and $r\in\pi_*R$ as above, there is a fiber sequence
    \begin{equation*}
        \Kloc(A)\rightarrow\Kloc(R)\rightarrow\Kloc(R[r^{-1}]),
    \end{equation*}
    of spectra, where $A=\End_R(R/r)^{\mathrm{op}}$.
\end{theorem}

The fiber of the map $\Kloc(\BP{n})\rightarrow\Kloc(\Erm(n))$ has been considered before
in Barwick's work~\cite{barwick-highercats}*{Example~11.16}, where it is observed that the
fiber is the $K$-theory of $v_n$-nilpotent $\BP{n}$-modules. One of the main contributions
of this paper is to identify the fiber as the $K$-theory of an $\EE_1$-ring.
This is the special case of Theorem \ref{thm:abstractfiberseq} when $R =
\BP{n}$ and $r = v_n$. In this case, we write $\A{n-1} =
\End_{\BP{n}}(\BP{n}/v_n)^{\mathrm{op}}$; in particular, we have a natural
$\EE_1$-ring map $\BP{n-1} \we \BP{n}/v_n \to \A{n-1}$.

\begin{theorem}\label{thm:maincofiber}
    For all $n\geq 0$, there is an $\EE_1$-ring $\A{n-1}$ and a fiber sequence
    \begin{equation*}
        \Kloc(\A{n-1})\rightarrow\Kloc(\BP{n})\rightarrow\Kloc(\Erm(n))
    \end{equation*}
    of spectra. Moreover, the transfer map $\Kloc(\BP{n-1})\rightarrow\Kloc(\BP{n})$
    factors through the map $\Kloc(\BP{n-1})\rightarrow\Kloc(\A{n-1})$ induced by $\BP{n-1} \to \A{n-1}$.
\end{theorem}

Here is an outline of how we use Theorem~\ref{thm:maincofiber} to prove
Theorem~\ref{thm:intromain}. To begin, we show that the homotopy ring of $\A{n-1}$ for $n>0$ is
\begin{equation*}
    \pi_*\A{n-1}\iso\ZZ_{(p)}[v_1,\ldots,v_{n-1}]\otimes\Lambda_{\ZZ_{(p)}}\langle\epsilon_{1-2p^n}\rangle,
\end{equation*}
where $\epsilon_{1-2p^n}$ has degree $1-2p^n$.
Moreover, we show that if the question of Rognes has a positive answer, then the natural map
\begin{equation*}
    \K(\BP{n-1})\rightarrow\K(\A{n-1})
\end{equation*}
must be an equivalence.

We use a rational trace argument to compare the $K$-theories of $\BP{n-1}$ and $\A{n-1}$.
After rationalization, we show that both
$\BP{n-1}_\QQ=\Hrm\QQ\otimes_\SS\BP{n-1}$ and
$\A{n-1}_\QQ=\Hrm\QQ\otimes_\SS\A{n-1}$ admit $\EE_\infty$-ring structures. Despite the dearth
of computational techniques for the $K$-theory of nonconnective ring spectra, the fact that
$\pi_*\BP{n-1}_\QQ$ and $\pi_*\A{n-1}_\QQ$ are both graded-commutative polynomial algebras allows us
to use trace methods to construct many classes in positive degree in $\Kloc(\A{n-1})$ that cannot come
from $\Kloc(\BP{n-1})$. 

To construct these classes, we study the commutative diagram
\begin{equation*}
    \xymatrix{
        \BGL_1(\BP{n-1})\ar[r]\ar[d]    &
    \Omega^\infty\K(\BP{n-1})\ar[r]\ar[d]    & \Omega^\infty\HH^{\Hrm\QQ}\left(\BP{n-1}_\QQ\right)\ar[d]\\
        \BGL_1(\A{n-1})\ar[r]    &   \Omega^\infty\K(\A{n-1})\ar[r]    &
    \Omega^\infty\HH^{\Hrm\QQ}\left(\A{n-1}_\QQ\right)\\
    }
\end{equation*}
of units and trace maps. A Hochschild--Kostant--Rosenberg-type isomorphism yields the
identification
\begin{equation}\label{eq:introthh}
    \HH^{\Hrm\QQ}_*\left(\A{n-1}_\QQ\right)\iso\QQ[v_1,\ldots,v_{n-1},\delta_{2-2p^n}]\otimes\Lambda_\QQ\langle\sigma_1,\ldots,\sigma_{n-1},\epsilon_{1-2p^n}\rangle,
\end{equation}
where the degree of $\sigma_i$ is $2p^i-1$. Moreover,
$\HH^{\Hrm\QQ}_*(\BP{n-1}_\QQ)\rightarrow\HH^{\Hrm\QQ}_*(\A{n-1}_\QQ)$ is the inclusion of
the subalgebra generated by the $v_i$ and $\sigma_i$ classes. Finally, we
compute the effect in homotopy of the compositions
$\BGL_1(\BP{n-1})\rightarrow\HH^{\Hrm\QQ}\left(\BP{n-1}_\QQ\right)$
and $\BGL_1(\A{n-1})\rightarrow\HH^{\Hrm\QQ}\left(\A{n-1}_\QQ\right)$ to prove
the following result.

\begin{theorem}
    If $x=v_1^{a_1}\cdots v_{n-1}^{a_{n-1}}\epsilon_{1-2p^n}$ is a monomial in $\pi_*\A{n-1}$
    of positive total degree, then the class
    \begin{equation*}
        v_1^{a_1}\cdots v_{n-1}^{a_{n-1}}\delta_{2-2p^n}+\sum_{i=1}^{n-1}a_i v_1^{a_1}\cdots v_i^{a_i-1}\cdots
        v_{n-1}^{a_{n-1}}\sigma_i\epsilon_{1-2p^n}
    \end{equation*}
    is in the image of $\K(\A{n-1})\rightarrow\HH^{\Hrm\QQ}(\A{n-1}_\QQ)$ and not
    in the image of $\HH^{\Hrm\QQ}(\BP{n-1}_\QQ)\rightarrow\HH^{\Hrm\QQ}(\A{n-1}_\QQ)$.
\end{theorem}

From this fact we immediately obtain Theorem~\ref{thm:intromain}. It is also clear why this
method does not contradict the known cases $n=0$ and $n=1$ of Rognes'
question. Indeed, in those cases there are no such monomials of positive total degree.

\begin{remark} Building on work of
    Rognes~\cites{rognes-2primary,rognes-whitehead},
    Blumberg and Mandell also give
    in~\cites{blumberg-mandell-ksphere,blumberg-mandell-tp} another approach to the
    computation of the $K$-groups of the sphere, which completely determines
    the homotopy type of the fiber of $\K(\SS)\rightarrow\K(\ZZ)$
    in terms of the $K$-groups of $\ZZ$, the homotopy groups of $\mathds{C}P^{\infty}_{-1}$, and the stable homotopy groups of spheres.
\end{remark}

\paragraph{Outline.}
Sections~\ref{sec:morita} and~\ref{sec:trace} contain our theorem on
localization sequences arising from inverting elements in $\EE_1$-rings and the trace
machinery we will use. Section~\ref{sec:calculations} provides a concrete method for computing the
trace map involving K\"ahler differentials, in some cases. The $\EE_1$-ring structures on
$\BP{n}$ are described briefly in Section~\ref{sec:tbps}. In Section~\ref{sec:obstruction}
we construct the $\EE_\infty$-ring structures on $\BP{n-1}_\QQ$ and $\A{n-1}_\QQ$.
Finally, in Section~\ref{sec:arc}, we give the proof of Theorem~\ref{thm:intromain},
resolving in the negative Rognes' question for $n>1$.

\paragraph{Notation.}
As a matter of convention, and unless noted otherwise, we will use $\infty$-categories throughout,
following Lurie's approach to stable homotopy theory developed in~\cite{ha}. We will speak of
$\EE_n$-rings, as opposed to $\EE_n$-ring spectra, of $\EE_1$-algebras over
$\EE_n$-rings for $n>1$, and of right modules, as opposed to right module spectra. If $\Cscr$ is an
$\infty$-category, we will write either $\Cscr(x,y)$ or $\map_{\Cscr}(x,y)$ for the space of
maps between two objects $x,y\in\Cscr$. If $\Cscr$ is in addition stable, we will write
$\Map_{\Cscr}(x,y)$ for the mapping spectrum. In the important case where $\Cscr=\Mod_A$,
the stable $\infty$-category of right $A$-modules for an $\EE_1$-ring $A$, we write
$\map_A(x,y)$ and $\Map_A(x,y)$ for the mapping space and spectrum.

\paragraph{Acknowledgments.}
We would like to thank Andy Baker, Bj{\o}rn Dundas, John Greenlees, Owen Gwilliam, Mike Hopkins,
Ayelet Lindenstrauss, Charles Rezk, Christian Schlichtkrull, and Sean Tilson for conversations about this work.
This project emerged while all three
authors participated in the 2014 Algebraic Topology Program at MSRI, and we thank both the
institute for its hospitality and the organizers of the program for creating such a
stimulating environment. The first named author would like to give special thanks for
Yank{\i} Lekili who asked him an apparently unrelated question which eventually led to the discovery
of a sign problem obstructing progress. Special thanks are reserved for Clark Barwick, Andrew Blumberg,
Mike Mandell, and John Rognes for their comments on an early draft. Finally,
we thank the referees for their careful reading of the text.

\section{The $K$-theory fiber of a localization of rings}\label{sec:morita}

In this section, we introduce algebraic $K$-theory and prove a theorem which describes the
fiber in $K$-theory of a localization of an $\EE_1$-ring in certain cases. Note that we
follow Lurie~\cite{ha} in terminology wherever possible. In particular,
using~\cite{bgt1}, we will view
connective algebraic $K$-theory (denoted $\K)$ and nonconnective algebraic $K$-theory (denoted
$\Kloc$) as a functor defined on small stable $\infty$-categories. There is no
substantive difference between this approach and the approach via Waldhausen categories:
see~\cite{bgt1}*{Section 7.2}.

\subsection{$K$-theory}

We start by introducing some terminology about small stable $\infty$-categories.

\begin{definition}
    \begin{enumerate}
        \item   A small stable $\infty$-category $\Cscr$ is \emph{idempotent complete} if it
            is closed under summands. The $\infty$-category of small stable idempotent complete
            $\infty$-categories and exact functors between them is denoted by $\Cat_{\infty}^{\perf}$.
        \item   A sequence $\Cscr\xrightarrow{f}\Dscr\xrightarrow{g}\EEcr$ in
            $\Cat_{\infty}^{\perf}$ is \emph{exact} if the composite $\Cscr\rightarrow\EEcr$
            is zero, $\Cscr\rightarrow\Dscr$ is fully faithful, and
            $\Dscr/\Cscr\rightarrow\EEcr$ is an equivalence. Note that the cofiber is taken
            in $\Cat_{\infty}^{\perf}$ and is the idempotent completion of the usual
            Verdier quotient.
        \item   Such a sequence is \emph{split-exact} if moreover there exist right adjoints
            $f_\rho\colon\Dscr\rightarrow\Cscr$ and $g_\rho\colon\EEcr\rightarrow\Dscr$ such that
            $f_\rho\circ f\we\id_\Cscr$ and $g\circ g_\rho\we\id_\EEcr$.
        \item   Let $\Sp$ denote the $\infty$-category of spectra, as defined
            in~\cite{ha}*{Section~1.4}. An \emph{additive invariant} of small stable $\infty$-categories is a
            functor $F\colon\Cat_\infty^\perf\rightarrow\Sp$ that takes split-exact sequences to
            split fiber sequences of spectra.
        \item   A \emph{localizing invariant} of small stable $\infty$-categories is a
            functor $F\colon\Cat_\infty^\perf\rightarrow\Sp$ that takes exact sequences to fiber
            sequences of spectra.
    \end{enumerate}
\end{definition}

To connect exact sequences in $\Cat_\infty^\perf$ to localization, let
$\Mod_\Cscr=\Fun^{\ex}(\Cscr^{\op},\Sp)$, the stable presentable $\infty$-category of right
$\Cscr$-modules in spectra. An exact sequence in $\Cat_\infty^\perf$ then gives rise by
left Kan extensions to an exact sequence
\begin{equation*}
    \Mod_\Cscr\rightarrow\Mod_\Dscr\rightarrow\Mod_\EEcr,
\end{equation*}
and the functor $\Mod_\Dscr\rightarrow\Mod_\EEcr$ is a localization, in the sense that its
right adjoint is fully faithful. Moreover, if $L\colon\Mod_{\Dscr}\rightarrow\Mod_{\EEcr}$ is a
localization such that the kernel is generated by the $\infty$-category $\Cscr$ of objects
$x$ of $\Dscr$ such that $L(x)\we 0$, then
\begin{equation*}
    \Cscr\rightarrow\Dscr\rightarrow\EEcr
\end{equation*}
is an exact sequence in $\Cat_\infty^\perf$.

\emph{Additive, or connective, $K$-theory} is an additive invariant
$\K\colon\Cat_\infty^\perf\rightarrow\Sp$, see~\cite{bgt1}*{Section 7}
and~\cite{bgt-endomorphisms}*{Section~2}. When $\Cscr$
is the $\infty$-category associated to a Waldhausen category by hammock localization, $\K(\Cscr)$
agrees with Waldhausen $K$-theory by~\cite{bgt-endomorphisms}*{Theorem~2.5}.

\emph{Localizing, or nonconnective, $K$-theory} is a localizing invariant
$\Kloc\colon\Cat_\infty^\perf\rightarrow\Sp$. It is defined in~\cite{bgt1}*{Section
9}. The idea goes back to Bass. It
agrees with the $K$-theory of Thomason and Trobaugh for the relevant cases of
schemes. By construction, there is a natural equivalence
$\K(\Cscr)\rwe\tau_{\geq 0}\Kloc(\Cscr)$, where $\tau_{\geq 0}\Kloc(\Cscr)$ is
the connective cover.

The question of Rognes is as stated about additive $K$-theory, but it will prove
easier to pass first through nonconnective $K$-theory, essentially for the results about
localizations of $\EE_1$-rings, to which we now turn.

\begin{definition}
    If $A$ is an $\EE_1$-ring, then $\K(A)$ and $\Kloc(A)$ are defined as the connective and
    nonconnective $K$-theories of $\Mod_A^\omega$, the small stable $\infty$-category of
    compact right $A$-modules.
\end{definition}

\begin{definition}
    Let $\Cscr$ be a small stable idempotent complete $\infty$-category. We will say that a compact object $M$ of $\Cscr$
    \emph{generates} $\Cscr$ if $\Map_{\Mod_\Cscr}(M,N)\we 0$ implies that $N\we 0$
    for $N$ in $\Mod_\Cscr$. Here $\Map_{\Mod_\Cscr}(M,N)$ denotes the mapping spectrum from $M$ to $N$.
\end{definition}

The following theorem is the main theorem of Morita theory for small stable
$\infty$-categories.

\begin{theorem}[Schwede--Shipley~\cite{schwede-shipley}*{Theorem~3.3}]
    Suppose that $\Cscr$ is a small idempotent complete stable $\infty$-category
    generated by an object $M$. If $A=\End_{\Cscr}(M)^\mathrm{op}$, then $\Mod_A\we\Mod_\Cscr$ and
    hence $\Mod_A^\omega\we\Cscr$.
\end{theorem}

In this situation, we will also say that $M$ is a compact generator of $\Mod_\Cscr$ (remembering that $\Cscr\we\Mod_{\Cscr}^{\omega}$).

\begin{corollary}\label{cor:kmor}
    In the situation of the theorem, $\K(A)\we\K(\Cscr)$ and $\Kloc(A)\we\Kloc(\Cscr)$.
\end{corollary}

We will use heavily a map from $\BGL_1(A)$ to
$\Omega^\infty\K(A)\we\Omega^\infty\Kloc(A)$, the
algebraic space of $A$.

\begin{definition}
    The group of units of an $\EE_1$-ring $A$ is defined as the homotopy pullback in the
    square
    \begin{equation*}
        \xymatrix{
            \GL_1(A)\ar[r]\ar[d]   &   \Omega^\infty A\ar[d]\\
            \pi_0A^\times\ar[r]     &   \pi_0A.
        }
    \end{equation*}
    Multiplication induces a grouplike $\EE_1$-algebra structure on the space $\GL_1(A)$, which we can
    deloop to obtain $\BGL_1(A)$.
\end{definition}

Unwinding the approach of~\cite{bgt1}*{Section
7}, we see that there is an equivalence
\begin{equation*}
    \Omega^\infty\K(A)\we\Omega|\left(S_\bullet\Mod_A^\omega\right)^{\eq}|,
\end{equation*}
where $S_\bullet\Mod_A^\omega$ is the simplicial $\infty$-category whose
$\infty$-category of $n$-simplices $S_n\Mod_A^\omega$ is equivalent to the functor
category $\mathrm{Fun}([n-1],\Mod_A^\omega)$, and where $\Mod_A^\omega$ denotes the
$\infty$-category of compact objects in $\Mod_A$. Here, $[n-1]$ denotes the
nerve of the graph $0\rightarrow 1\rightarrow\cdots\rightarrow n-1$, and
$[-1]=\emptyset$.
% As indicated in the notation, this is
% an infinite loop space, and we let $\K(A)$ be the spectrum obtained from delooping.

\begin{definition}
    The \emph{units map} $\BGL_1(A)\rightarrow \Omega^{\infty}\K(A)$ is induced by the map from $\BGL_1(A)$ into
    the $1$-skeleton $\map([0],\Mod_A^\omega)$ of $S_\bullet\Mod_A^\omega$ given by the subspace of functors
    $[0]\rightarrow\Mod_A^\omega$ sending $0$ to $A$.
\end{definition}

In fact, similar reasoning defines a map from $\Brm\Aut_A(P)$ to $\Omega^\infty\K(A)$
for any compact right $A$-module $P$. There is an induced map
\begin{equation*}
    \coprod_{\text{$P$ f.g. projective}}\Brm\Aut_A(P)\rightarrow\Omega^\infty\K(A)
\end{equation*}
of $\EE_\infty$-spaces, where the $\EE_\infty$-space structure on the left-hand
side is induced by direct sum of finitely generated projective modules\footnote{Just as in ordinary algebra, a projective
$A$-module is a retract of a free $A$-module $A^I$ for some set $I$.
See~\cite{ha}*{Section~7.2.2}.}. Since $\Omega^\infty\K(A)$
is grouplike, the map factors through the group completion map
$\coprod_{\text{$P$ f.g.
projective}}\Brm\Aut_A(P)\rightarrow\Omega\Brm\left(\coprod_{\text{$P$ f.g.
projective}}\Brm\Aut_A(P)\right)$.

\begin{theorem}[\cite{ekmm}*{Theorem VI.7.1}]\label{thm:groupcompletion}
    If $A$ is connective, the induced map
    $$\Omega\Brm\left(\coprod_{\text{$P$ f.g. projective}}\Brm\Aut_A(P)\right)\rightarrow\Omega^\infty\K(A)$$ is an equivalence.
\end{theorem}

Note that the presentation of~\cite{ekmm} uses only finitely generated free
$A$-modules. The proof of the result stated here is no different, and has the correct
$\K_0$ group.

\subsection{Localizations of $\EE_1$-rings}\label{sub:loce1}

We are interested in identifying the fiber of $\Kloc(A)\rightarrow\Kloc(B)$,
where $A$ is an $\EE_1$-ring and $B$ is a localization of $A$
obtained by inverting some elements of $\pi_*A$. It is not always possible to invert
elements in noncommutative rings. One way to ensure that it is possible is to impose the
following Ore condition.

\begin{definition}
    Let $A$ be an $\EE_1$-ring, and let $S\subset\pi_*A$ be a multiplicatively closed set of homogeneous elements.
    Then, $S$ satisfies the right Ore condition if (1) for every pair of elements
    $x\in\pi_*A$ and $s\in S$ there exist $y\in\pi_*A$ and $t\in S$ such that $xt=sy$, and
    (2) if $sx=0$ for some $x\in\pi_*A$ and $s\in S$, then there exists $t\in S$ such that $xt=0$.
\end{definition}

The conditions are particularly easy to verify when $\pi_*A$ is in fact a graded-commutative
ring. The following proposition collects several facts due to Lurie about the localization of an
$\EE_1$-ring at a multiplicative set satisfying the right Ore condition.

\begin{proposition}[Lurie]\label{prop:ore}
    Let $A$ be an $\EE_1$-ring, and let $S\subseteq\pi_*A$ be a multiplicatively closed set
    satisfying the right Ore condition. Then,
    \begin{enumerate}
        \item   there is an $\EE_1$-ring $S^{-1}A$ and an $\EE_1$-ring map $A\rightarrow
            S^{-1}A$, where $\pi_*S^{-1}A\iso S^{-1}\pi_*A$;
        \item   the induced functor $\Mod_A\rightarrow\Mod_{S^{-1}A}$ is a localization;
        \item   the induced functor $\Mod_A\rightarrow\Mod_{S^{-1}A}$ has as fiber the full
            subcategory $\Mod_A^{\mathrm{Nil}(S)}$ of $S$-nilpotent $A$-modules, i.e., those right
            $A$-modules $M$ such that for every homogeneous element $x\in\pi_*M$, there exists $s\in S$ such that
            $xs=0$.
    \end{enumerate}
\end{proposition}

\begin{proof}
    For (1), see~\cite{ha}*{Proposition 7.2.4.20} and~\cite{ha}*{Remark 7.2.4.26}. This
    identifies $S^{-1}A$ with an $\EE_1$-ring $A[S^{-1}]$, which is by definition the
    endomorphism algebra of the image of $A$ under a functor
    $\Mod_A\rightarrow\Mod_A^{\mathrm{Loc}(S)}$ and serves as a compact generator for $\Mod_A^{\mathrm{Loc}(S)}$. This functor is a localization with kernel
    the $S$-nilpotent $A$-modules
    by~\cite{ha}*{Proposition~7.2.4.17}. This proves (2) and (3).
\end{proof}

The result of this section identifies the fiber in $\K$-theory of an Ore localization
generated by a single element. The role of $S$-nilpotent modules has also been studied
by Barwick~\cite{barwick-highercats}*{Proposition 11.15}. Our identification below of a
special $S$-nilpotent module can be
viewed as a partial generalization of the work of Neeman and Ranicki~\cite{neeman-ranicki} to noncommutative
$\EE_1$-rings.

\begin{theorem}\label{thm:kfiber}
    Let $A$ be an $\EE_1$-ring, and let $S=\{1,r,r^2,\ldots\}$ be a multiplicatively closed
    subset of $\pi_*A$
    satisfying the right Ore condition and generated by a single homogeneous element
    $r\in\pi_dA$. If $M$ denotes the cofiber of the map $\Sigma^dA\xrightarrow{r} A$ of right
    $A$-modules, then there is a fiber sequence
    \begin{equation*}
        \Kloc\left(\End_A(M)^{\mathrm{op}}\right)\rightarrow\Kloc(A)\rightarrow\Kloc(S^{-1}A)
    \end{equation*}
    of spectra.
\end{theorem}

\begin{proof}
    By Proposition~\ref{prop:ore}, there is a localization sequence
    \begin{equation*}
        \Mod_A^{\mathrm{Nil}(S)}\rightarrow\Mod_A\rightarrow\Mod_{S^{-1}A}.
    \end{equation*}
    It follows that after taking compact objects we obtain an exact sequence
    \begin{equation*}
        \Mod_A^{\mathrm{Nil}(S),\omega}\rightarrow\Mod_A^{\omega}\rightarrow\Mod_{S^{-1}A}^\omega
    \end{equation*}
    of small stable $\infty$-categories, where
    $\Mod_A^{\mathrm{Nil}(S),\omega}$ is equivalent to the stable $\infty$-category of
    $S$-nilpotent compact right $A$-modules. By definition of localizing $K$-theory and by
    Corollary~\ref{cor:kmor}, it suffices
    to show that $M$ is a compact generator of $\Mod_A^{\mathrm{Nil}(S)}$. Since $M$ is
    built in finitely many steps from $A$ by taking cofibers, it is an object of
    $\Mod_A^\omega$. Moreover, the localization of $M$ is clearly zero since $r$ is a
    unit in $S^{-1}A$.

    By~\cite{ha}*{Proposition~7.2.4.14},
    an $A$-module $N$ is $S$-local if and only if $\Map_A(A/s,N)\we 0$ for all
    $s\in S$, where $A/s$ denotes the cofiber of
    $\Sigma^{|s|}A\xrightarrow{s}A$. Under the
    hypothesis above, we can replace this condition by the condition that
    $\Map_A(A/r,N)\we 0$. With $M = A/r$, we will now argue that if $\Map_A(M,N)\we
    0$ and $N$ is $S$-nilpotent, then $N\we 0$. This will show that $M$ is a compact generator of
    $\Mod_A^{\mathrm{Nil}(S)}$. Consider the cofiber sequence
    $\Sigma^{d}A\xrightarrow{r} A\rightarrow M$ of right $A$-modules. Under the assumption that
    $\Map_A(M,N)\we 0$, the induced map
    \begin{equation*}
        N\we\Map_A(A,N)\xrightarrow{r}\Map_A(\Sigma^{d}A,N)\we\Sigma^{-d}N
    \end{equation*}
    of left $A$-modules is an equivalence. Since $N$ is $S$-nilpotent, requirement (1) of
    the Ore implies that every homogeneous element of $\pi_*N$ is annihilated by
    multiplication \emph{on the left} by a power of $r$. Hence, $N\we 0$.
\end{proof}

The localization at the heart of the Rognes' question is the map
$\BP{n}\rightarrow\Erm(n)$, where $\Erm(n)$ is obtained from $\BP{n}$ by inverting $v_n$.
Since $\pi_*\BP{n}$ is in fact a commutative ring, the right Ore condition is trivially satisfied. It
follows from the theorem that if we can compute the endomorphism ring of the cofiber of
$v_n\colon\Sigma^{2p^n-2}\BP{n}\rightarrow\BP{n}$, then we obtain an $\EE_1$-ring whose
$K$-theory can be analyzed and compared to the $K$-theory of $\BP{n-1}$. We will return to
the truncated Brown--Peterson spectra after introducing the computational techniques we will use.
 
We briefly discuss two different directions in which one can generalize the theorem.

First of all one can consider localizations $A\rightarrow S^{-1}A$ where $S$ is
generated by a left-regular sequence of homogeneous elements
$(r_1,\ldots,r_n)$. We give the statement for $n=2$.
Suppose that $\mathrm{deg}(r_i)=d_i$ for $i=1,2$ and suppose that $r_1$ and $r_2$ commute
up to homotopy when acting by multiplication on the left. That is, suppose that in
$\Ho(\Mod_A)$ the left-hand square in 
\begin{equation*}
    \xymatrix{
        \Sigma^{d_1+d_2}A\ar[r]^{r_1}\ar[d]^{r_2}   &
    \Sigma^{d_2}A\ar[d]^{r_2}\ar[r] & \Sigma^{d_2}M_1\ar[r]\ar@{.>}[d]^{r_2} & \Sigma^{d_1 + d_2 +1}A \ar[d]^{\Sigma r_2} \\
        \Sigma^{d_1}A\ar[r]^{r_1}                   &   A\ar[r]                         & M_1 \ar[r] & \Sigma^{d_1 +1}A
    }
\end{equation*}
commutes, where $M_1$ is
the cofiber of $r_1$ acting by multiplication on the left as a right $A$-module map. In this case, the triangulated category axioms assert that there is an extension
of this diagram to the right by a map $r_2\colon\Sigma^{d_2}M_1\rightarrow
M_1$. The map $r_2$ is
an example of a Toda bracket and need not be unique. We will say that
$(r_1,r_2)$ is a \emph{left-regular sequence} if the multiplicatively closed
set generated by $r_1$ and $r_2$ satisfies the right Ore condition, if $r_1$ acts injectively on the
homotopy of $A$, if the square above commutes, and
if for \emph{some}
choice of fill $r_2\colon\Sigma^{d_2}M_1\rightarrow M_1$ the map in homotopy is injective.
Note that a priori $(r_1,r_2)$ may be a left-regular sequence while $(r_2,r_1)$
is not.

\begin{proposition}
    Suppose that $A$ is an $\EE_1$-ring and that $S$ is the multiplicative set
    generated by a left-regular sequence $(r_1,r_2)$. Let $M_2$ be the cofiber
    of $\Sigma^{d_2}M_1\xrightarrow{r_2}M_1$ for some choice of map $r_2$ as
    above which is injective in homotopy. Then, there is a fiber sequence
    $$\Kloc(\End_A(M_2)^{\op})\rightarrow\Kloc(A)\rightarrow\Kloc(S^{-1}A)$$ of
    spectra.
\end{proposition}

\begin{proof}
    We prove that $M_2$ is a compact generator of the $S$-nilpotent spectra. Suppose
    that $N$ is $S$-nilpotent and that $\Map_A(M_2,N)\we 0$. As above one then sees that $r_2$
    induces an equivalence
    \begin{equation*}
        \Map_A(M_1,N)\rightarrow\Map_A(\Sigma^{d_2}M_1,N).
    \end{equation*}
    On the other hand, $r_2$ acts nilpotently on the homotopy groups of
    $N$. Consider the Ext spectral sequence
    \begin{equation*}
        \Eoh_2^{s,t}=\Ext^s_{\pi_*A}(\pi_*M_{1},\pi_*N)_t\Rightarrow\pi_{t-s}\Map_A(M_{1},N).
    \end{equation*}
    By the right Ore condition and the left-regularity of $r_1$, it follows that the right
    $\pi_*A$-module $\pi_*M_1$ has graded projective dimension $1$. Therefore, the filtration on
    the abutment is finite. But, $r_2$ must then act nilpotently on $\pi_*\Map_A(M_1,N)$, which
    implies that $\Map_A(M_1,N)\we 0$. The argument concludes as in the proof of the theorem.
\end{proof}

Appropriate similar hypothesis can be imposed on localizations of more elements to ensure the existence of
fiber sequences in $K$-theory with fiber given by the $K$-theory of an endomorphism
algebra.

In the second direction we can consider localizations of connective
$\EE_\infty$-rings or more generally of derived schemes (see~\cite{ag}).

\begin{theorem}
    Let $X$ be a quasi-compact and quasi-separated derived scheme and $U\subseteq X$ a
    quasi-compact Zariski open subscheme with complement $Z$. Then, there is a fiber sequence
    \begin{equation*}
        \Kloc(A)\rightarrow\Kloc(X)\rightarrow\Kloc(U),
    \end{equation*}
    where $\Kloc(A)\we\Kloc(\text{$X$ on $Z$})$ is the $K$-theory of an $\EE_1$-ring
    spectrum $A=\End_X(M)^{\op}$.
\end{theorem}

\begin{proof}
    In this situation, Antieau and
    Gepner proved in~\cite{ag}*{Proposition 6.9} that there is a single compact
    object $M$ that generates the fiber of
    \begin{equation*}
        \Mod_X\rightarrow\Mod_U.
    \end{equation*}
    This completes the proof.
\end{proof}

Typically, $M$ can be taken to be a kind of generalized Koszul complex. For example, when
$X=\Spec R$ is affine, and when $U$ is the open set defined by inverting
$r_1,\ldots,r_n$ all of degree zero, then $M$ can be taken as
\begin{equation*}
    \bigotimes_{i=1}^n \cofib\left(R\xrightarrow{r_i}R\right),
\end{equation*}
which is precisely the Koszul complex when $R$ is an Eilenberg--MacLane spectrum and the
$r_i$ form a regular sequence.

\section{Hochschild homology and trace}\label{sec:trace}

Throughout this section $R$ will denote an $\EE_\infty$-ring.

\subsection{Hochschild homology}

\begin{definition}
    Let $A$ be an $\EE_1$-$R$-algebra and $M$ an $A$-bimodule in $\Mod_R$.
    The Hochschild homology $\HH^R(A,M)$ of $A$ over $R$ with coefficients in $M$ is the
    geometric realization of the cyclic bar construction $\Bcyc_\bullet(A/R,M)$,
    the simplicial spectrum with level $n$ the $R$-module
    \begin{equation*}
        \Bcyc_n(A/R,M)=M\otimes_R A^{\otimes n}.
    \end{equation*}
    The face and degeneracy maps are given by the same formulas as in ordinary Hochschild
    homology, see for instance~\cite{shipley-hochschild}*{Definition~4.1.2}. If $M=A$ with
    its natural bimodule structure then we typically write $\Bcyc_\bullet A$ for the
    simplicial spectrum and $\HH^R(A)$ for its geometric realization,
    omitting reference to $R$ when no confusion can occur. When
    $R=\Hrm\pi_0R$ is an Eilenberg--MacLane spectrum, our (implicit) use of the derived
    tensor product implies that this definition can disagree with ordinary Hochschild
    homology even when $M$ and $A$ are themselves Eilenberg--MacLane. It follows from the Tor spectral sequence that no discrepancy occurs
    if their homotopy groups are flat over $\pi_0R$.
\end{definition}

We will need the following Hochschild--Kostant--Rosenberg theorem for our later computations.
Our argument follows the proof of a similar theorem of McCarthy and
Minasian~\cite{mccarthy-minasian}*{Theorem~6.1}. Let $\CAlg_R$ denote the
$\infty$-category of $\EE_\infty$-$R$-algebras, and let
$\Sym_R\colon\Mod_R\rightarrow\CAlg_R$ denote the symmetric algebra functor, which
gives the free $\EE_\infty$-$R$-algebra on an $R$-module $M$. It is left
adjoint to the forgetful functor $\Mod_R\leftarrow\CAlg_R$. Note also that if
$R\rightarrow S$ is a map of $\EE_\infty$-rings, then there is a natural
equivalence $S\otimes_R\Sym_R(M)\we\Sym_S(S\otimes_R M)$.

\begin{theorem}\label{thm:symmetrichkr}
    Let $R$ be an $\EE_\infty$-ring, and let $M$ be an $R$-module. The
    Hochschild homology of the free $\EE_\infty$-$R$-algebra $S=\Sym_{R}M$ over $R$ is
    given by
    \begin{equation*}
        \HH^{R}(S)\we \Sym_S(S\otimes_{R}\Sigma M)
    \end{equation*}
    as an $\EE_\infty$-$S$-algebra, where $S\otimes_R\Sigma M$ is the tensor
    product of $R$-modules.
\end{theorem}

\begin{proof}
    By McClure--Schw\"anzl--Vogt~\cite{mcclure-schwanzl-vogt}, there is an
    equivalence
    \begin{equation*}
        \HH^{R}(S)\we S^1\otimes_{R}S=S^1\otimes_{R}\Sym_{R}M,
    \end{equation*}
    where $S^1\otimes_R S$ refers to the simplicial structure on the category of $\EE_\infty$-$R$-algebras.
    As the symmetric algebra functor $\Sym_R\colon\Mod_R\rightarrow\CAlg_R$ is a left adjoint, it commutes with all
    all small colimits. Since the symmetric monoidal structure on $\CAlg_R$ is
    the cocartesian symmetric monoidal structure, $\Sym_R$ is symmetric
    monoidal for the cocartesian monoidal structure on $\Mod_R$. Thus, there is
    a natural equivalence
    $\Sym_R(M\oplus\Sigma M)\we\Sym_R(M)\otimes_R\Sym_R(\Sigma M)$. It follows that
    \begin{align*}
        S^1\otimes_{R}\Sym_{R}M&\we\Sym_{R}(S^1\otimes_{R}M)\\
        &\we\Sym_R(M\oplus\Sigma M)\\
        &\we\Sym_R(M)\otimes_R\Sym_R(\Sigma M)\\
        &\we\Sym_{\Sym_RM}(\Sym_RM\otimes_R\Sigma M),
    \end{align*}
    as desired.
\end{proof}

We will only apply this HKR-type theorem in the case where the base ring $R$ is $\Hrm\QQ$ and $M$ is
a compact $\Hrm\QQ$-module. In this case, the result is especially simple and
the reader can compare this result to the main result~\cite{hkr}.

\begin{corollary}\label{cor:qhkr}
    Let $M$ be a compact $\Hrm\QQ$-module with a basis for the even homology given by
    homogeneous elements $x_1,\ldots,x_m$ and for odd homology by homogeneous elements
    $y_1,\ldots,y_n$. Let $S=\Sym_{\Hrm\QQ}M$, so that
    $\pi_*S\iso\QQ[x_1,\ldots,x_m]\otimes\Lambda_\QQ\langle y_1,\ldots,y_n\rangle$. Then,
    \begin{equation*}
        \pi_*\HH^{\Hrm\QQ}(S)\iso\QQ[x_1,\ldots,x_m,\delta(y_1),\ldots,\delta(y_n)]\otimes\Lambda_\QQ\langle
        y_1,\ldots,y_n,\delta(x_1),\ldots,\delta(x_m)\rangle,
    \end{equation*}
    where the $\delta(y_i)$ and $\delta(x_i)$ are elements in degrees $|y_i|+1$
    and $|x_i|+1$, respectively, induced from a map $\Sigma
    M\rightarrow \HH^{\Hrm\QQ}(S)$.
\end{corollary}

\subsection{The trace map}

B\"okstedt introduced a natural transformation $\K\rightarrow\HH^\SS$ from
additive $K$-theory to Hochschild homology over the sphere called the
topological Dennis trace
map, generalizing earlier work of Waldhausen~\cite{waldhausen-2} in the special case of
$K$-theory of spaces. We will refer to this simply as the trace map.
Note that $\HH^\SS$ is precisely what is usually termed topological Hochschild homology
and denoted by $\THH$. This map and its refinements (such as topological cyclic homology)
play a central role in the contemporary approach to computations of $\K$-groups.
See Hesselholt and Madsen's work~\cite{hesselholt-madsen} for an exemplary case.

There is another approach which rests on a definition of $\HH^\SS$ for arbitrary small
spectral categories or stable $\infty$-categories. The idea goes back to Dundas and
McCarthy~\cite{dundas-mccarthy-hochschild}; see~\cite{blumberg-mandell-thh}*{Definition 3.4}
for the definition for spectral categories.
Blumberg, Gepner, and Tabuada show that this definition defines a
localizing invariant of small idempotent complete $\infty$-categories. It follows that
$\HH_0^\SS(\SS)$ is naturally bijective to the group of natural transformations
$\Kloc\rightarrow\HH^\SS$ of localizing invariants. They give a conceptual identification of
Waldhausen's trace, showing in~\cite{bgt1}*{Theorem~10.6} that it is the unique such natural
transformation with the property that the composition
\begin{equation*}
    \SS\rightarrow\Kloc(\SS)\rightarrow\HH^\SS(\SS)\rwe\SS
\end{equation*}
is homotopic to the identity, where $\SS\rightarrow\Kloc(\SS)$ is the unit map in $K$-theory, and
$\HH^\SS(\SS)\rwe\SS$ is the augmentation, i.e., the inverse to the unit map $\SS\rwe\HH^\SS(\SS)$ in Hochschild
homology.

\begin{definition}
    If $G$ is a grouplike $\EE_1$-algebra in spaces, we let $\Bcyc_\bullet G$ be the
    simplicial space with $\Bcyc_n=G^{n+1}$ and the usual face and degeneracy maps. Writing
    $\Bcyc G$ for the geometric realization of $\Bcyc_\bullet G$, we obtain by~\cite{jones}*{Theorem~6.2} a model for the free loop space
    $\Lrm\Brm G$ of $\Brm G$. There is a natural fibration sequence $G\rightarrow\Bcyc G\rightarrow\Brm
    G$, and there is a section on the right $\Brm G\rightarrow\Bcyc G$ given by including
    $\Brm G$ as the constant loops. This map in fact exists even before geometric
    realizations as a map $c\colon\Brm_\bullet G\rightarrow\Bcyc_\bullet G$ which on level $n$ is
    described by
    \begin{equation*}
        c(g_1,\ldots,g_n)=(g_n^{-1}\cdots g_1^{-1},g_1,\ldots,g_n).
    \end{equation*}
    See~\cite{loday}*{Proposition~7.3.4, Theorem~7.3.11}.
    This concrete cycle-level description will be important below.
\end{definition}

The following proposition is well-known and appears in similar forms in work of
B\"okstedt--Hsiang--Madsen~\cite{bokstedt-hsiang-madsen} and
Schlichtkrull~\cite{schlichtkrull-units}*{Section 4.4}.

\begin{proposition}
    Let $A$ be an $\EE_1$-ring. The composition
    $$\Sigma^\infty\BGL_1(A)_+\rightarrow\K(A)\rightarrow\Kloc(A)\rightarrow\HH^\SS(A)$$
    is equivalent to
    \begin{equation*}
        \Sigma^\infty(\BGL_1(A))_+\xrightarrow{c}\Sigma^\infty(\Bcyc\GL_1(A))_+\we\Bcyc\Sigma^\infty\GL_1(A)_+\rightarrow\Bcyc
        A=\HH^\SS(A),
    \end{equation*}
    where $\Bcyc\Sigma^\infty\GL_1(A)_+\rightarrow\Bcyc A$ is induced by the map of
    $\EE_1$-rings $\Sigma^\infty\GL_1(A)_+\rightarrow A$ adjoint to the inclusion of units.
\end{proposition}

\begin{proof}
    Given an $\EE_1$-ring $A$, there is a natural diagram of spectra
    \begin{equation*}
        \xymatrix{
            \Sigma^\infty\BGL_1(\tau_{\geq 0}A)_+\ar[r]\ar[d]^{\we}  &   \K(\tau_{\geq 0}A)\ar[r]\ar[d]  & \Kloc(\tau_{\geq 0}A)\ar[r]\ar[d]   &   \HH^\SS(\tau_{\geq 0}A)\ar[d]\\
            \Sigma^\infty\BGL_1(A)_+\ar[r]  &   \K(A)\ar[r]  & \Kloc(A)\ar[r]   &   \HH^\SS(A),\\
        }
    \end{equation*}
    where $\tau_{\geq 0}A$ denotes the connective cover of $A$. In particular, to compute the
    restriction of the trace $\Kloc(A)\rightarrow\HH^\SS(A)$ to the classifying space of the units of $A$,
    we can assume
    that $A$ is connective. For connective rings, Waldhausen defined the trace map in the
    following way (see~\cite{bokstedt-hsiang-madsen}*{Section 5}). For each finitely generated projective right $A$-module $P$, look at the map
    \begin{equation*}
        \Brm\Aut_A(P)\xrightarrow{c}\Bcyc\Aut_A(P)\rightarrow\Omega^\infty\Bcyc\End_A(P)\rightarrow \Omega^\infty\HH^{\SS}(A),
    \end{equation*}
    where $\End_A(P)$ is the endomorphism algebra spectrum of $P$, and
    where the right-hand map is the usual trace map in Hochschild homology (defined as
    in~\cite{weibel-homological}*{Section 9.5}
    or~\cite{blumberg-mandell-thh}*{Theorem~5.12}). The right-hand
    map is an equivalence if $P$ is a faithful projective right $A$-module by Morita theory. In any case,
    these maps are compatible with direct sum, and hence they induce a well-defined map
    \begin{equation*}
        \coprod_{\text{$P$ f.g. projective}}\Brm\Aut_A(P)\rightarrow\Omega^\infty\HH^\SS(A)
    \end{equation*}
    which is compatible with the direct sum structure on the left. Therefore, they induce a map
    from the group completion, which is equivalent to $\Omega^\infty\K(A)$ by Theorem~\ref{thm:groupcompletion}, to
    $\Omega^\infty\HH^\SS(A)$. Using the adjunction between $\Omega^\infty$ and
    $\Sigma^\infty_+$ completes the proof.
\end{proof}

We will use homotopy classes in $\BGL_1(A)$ and prove that they are
non-zero in $\HH^\SS(A)$. This will have the consequence that they map to non-zero
$K$-theory classes. More specifically, we will construct non-zero homotopy
classes in the $1$-skeleton of $\BGL_1(A)$, and we will need to know that they survive to
non-zero classes in $\HH^\SS(A)$. For this, we will need to understand the B\"okstedt spectral
sequence converging to $\HH^\SS$, namely
\begin{equation*}
    \Eoh^2_{s,t}=\Hoh_s(\Delta^{\op},\pi_t\Bcyc_{\bullet}A)\Rightarrow\HH_{s+t}^\SS(A).
\end{equation*}
This is nothing other than the homotopy colimit spectral sequence for the geometric
realization of a simplicial spectrum. Moreover, there are natural isomorphisms
\begin{equation*}
    \Hoh_s(\Delta^{\op},\pi_t\Bcyc_\bullet A)\iso\Hoh_s(C_*(\pi_t\Bcyc_\bullet A))
\end{equation*}
for all $t$, where $C_*(\pi_t\Bcyc_\bullet A)$ is the unnormalized chain complex associated
to the simplicial abelian group $\pi_t\Bcyc_\bullet A$.

In particular, we would like to understand the map
$\pi_t\Brm_\bullet\GL_1(A)\rightarrow\pi_t\Bcyc_\bullet A$ of simplicial abelian groups in
terms of elements of these groups. The next proposition provides the necessary formula, but
we need to note a couple of easy lemmas first.

\begin{lemma}\label{lem:m1}
    If $G$ is a grouplike $\EE_1$-space, then the inverse map acts as $-1$ on $\pi_tG$
    for $t\geq 1$.
\end{lemma}

\begin{proof}
    This follows from the Eckmann--Hilton argument.
\end{proof}

Given two based spaces $X$ and $Y$, consider the Tor spectral sequence
$$\Eoh_2^{s,t}=\Tor_s^{\pi_*\SS}(\pi_*\Sigma^\infty X_+,\pi_*\Sigma^\infty
Y_+)_t\Rightarrow\pi_{s+t}(\Sigma^\infty X_+\otimes\Sigma^\infty Y_+),$$
with differentials $d_r$ of bidegree $(-r,r-1)$.
When $s=0$, we have $\pi_*\Sigma^\infty X_+\otimes_{\pi_*\SS}\pi_*\Sigma^\infty
Y_+$, and this contributes to the lowest part of the filtration on the
abutment because every differential vanishes on $\Eoh_r^{0,t}$.
In particular, given $a\in\pi_tX$ and $b\in\pi_tY$, the classes $a\otimes 1$
and $1\otimes b$ determine homotopy classes in $\pi_t(\Sigma^\infty
X_+\otimes\Sigma^\infty Y_+)$.

\begin{lemma}\label{lem:m2}
    Let $X$ and $Y$ be based spaces. Then, the map
    \begin{equation*}
        \pi_t(X\times Y)\rightarrow\pi_t^s(X\times Y)_+\rightarrow\pi_t(\Sigma^\infty
        X_+\otimes\Sigma^\infty Y_+)
    \end{equation*}
    sends $(a,b)$ to $a\otimes 1+1\otimes b$.
\end{lemma}

\begin{proof}
    If either $a$ or $b$ is zero, this is obvious. But, $(a,b)=(a,0)+(0,b)$.
\end{proof}

Similar considerations apply in
the following proposition. For the sake of clarity, write $u$ for the natural map
$\GL_1(A)\rightarrow\Omega^\infty A$, as well as the corresponding map
$\pi_t\GL_1(A)\rightarrow\pi_t A$, which is an isomorphism for $t\geq 1$.
 
\begin{proposition}\label{prop:homotopymap}
    The composition
    \begin{equation*}
        \pi_t\Brm_n\GL_1(A)\rightarrow\pi_t\Bcyc_{n}\GL_1(A)\rightarrow\pi_t\Bcyc_{n}\Sigma^\infty\GL_1(A)_+\rightarrow\pi_t\Bcyc_{n}A
    \end{equation*}
    sends $(g_1,\ldots,g_n)$ to
    \begin{equation*}
        \left(1\otimes u(g_1)\otimes 1\otimes\cdots\otimes 1\right)+\cdots+\left(1\otimes\cdots\otimes 1\otimes
        u(g_n)\right)-\left(u(g_1+\cdots+g_n)\otimes 1\otimes\cdots\otimes 1\right)
    \end{equation*}
    for $t\geq 1$.
\end{proposition}

\begin{proof}
    The first map sends $(g_1,\ldots,g_n)$ to $(-g_1-\cdots-g_n,g_1,\ldots,g_n)$ by
    Lemma~\ref{lem:m1} and the description of the inclusion of the constant loops at the
    simplicial level. The rest of the description then follows from Lemma~\ref{lem:m2} and
    the fact that $u\colon\GL_1(A)\rightarrow\Omega^\infty A$ factors as
    $\GL_1(A)\rightarrow\Omega^\infty\Sigma^\infty\GL_1(A)_+\rightarrow\Omega^\infty A$.
\end{proof}

\section{K\"ahler differentials}\label{sec:calculations}

In this section we study the composition
\begin{equation*}
    \Tr\colon\BGL_1(A)\rightarrow\Omega^\infty\K(A)\rightarrow\Omega^\infty\HH^{\SS}(A)\rightarrow\Omega^\infty\HH^{\Hrm\QQ}(A_\QQ)
\end{equation*}
in the special case where $A_\QQ=\Hrm\QQ\otimes_\SS A$ admits an $\EE_\infty$-ring structure.
We can then assume by~\cite{shipley-hz}*{Theorem~1.2} that $A_\QQ$ is a rational commutative
differential graded algebra, which we do in this section.

Recall (for example from~\cite{loday}*{Paragraph~5.42} in the connective case) that for a rational commutative dga $A$ the dg module $\Omega_{A/\QQ}$ of K\"ahler
differentials can be defined as the free $A$-module generated by symbols $da$ for the
homogeneous elements $a$ of $A$ modulo the relations
\begin{equation*}
    d(ab)=ad(b)+(-1)^{|a||b|}bd(a),
\end{equation*}
where the map $d\colon A\rightarrow\Omega_{A/\QQ}$ is the universal derivation of $A$ over $\QQ$.
We should remark that by using the cotangent complex one may obtain more information, but
this is not necessary in the present paper.

\begin{proposition}\label{prop:unider}
    Suppose that $A$ is an $\EE_1$-ring such that $A_\QQ$ admits an $\EE_\infty$-ring
    structure (compatible with the $\EE_1$-ring structure). Then, there is a natural map
    \begin{equation*}
        D\colon\HH^{\Hrm\QQ}(A_\QQ)\rightarrow\Sigma\Omega_{A_\QQ/\QQ}
    \end{equation*}
    such that the composition
    \begin{equation*}
        D\circ\Tr\colon\pi_{*-1}\GL_1(A)\rightarrow\K_*(A)\rightarrow\HH^\SS_*(A)\rightarrow\HH^{\Hrm\QQ}_*(A_\QQ)\rightarrow\pi_{*-1}\Omega_{A_\QQ/\QQ}
    \end{equation*}
    agrees with the composition of rationalization $A\rightarrow A_\QQ$ and the
    universal derivation $d\colon A_\QQ\rightarrow\Omega_{A_{\QQ}/\QQ}$ in homotopy
    groups in positive degrees.
\end{proposition}

\begin{proof}
    The
    Hochschild complex
    \begin{equation*}
        \cdots\rightarrow A^{\otimes 3}_\QQ\xrightarrow{\partial_0-\partial_1+\partial_2} A^{\otimes
        2}_\QQ\xrightarrow{\partial_0-\partial_1} A_\QQ\rightarrow 0
    \end{equation*}
    can be considered as a chain complex in chain complexes (or a double complex by changing
    the vertical differentials using the vertical sign trick). Since $A_\QQ$ is graded-commutative,
    $\partial_0-\partial_1=0$. Similarly,
    \begin{align*}
        (\partial_0-\partial_1+\partial_2)(a_0\otimes a_1\otimes a_2)&=a_0a_1\otimes a_2-a_0\otimes a_1a_2+(-1)^{|a_2|(|a_0|+|a_1|)}a_2a_0\otimes a_1\\
        &=a_0\left(a_1\otimes a_2-1\otimes a_1a_2+(-1)^{|a_2||a_1|}a_2\otimes a_1\right),
    \end{align*}
    where the sign conventions are the usual ones for the Hochschild complex of a
    differential graded algebra. See the material in~\cite{weibel}*{Section~9.9.1} on cyclic
    homology.

    Hence, the map $D\colon A_\QQ^{\otimes 2}\rightarrow\Omega_{A_\QQ/\QQ}$ defined by
    \begin{equation*}
        D(a_0\otimes a_1)=a_0d(a_1)
    \end{equation*}
    defines a quasi-isomorphism
    \begin{equation*}
        A_\QQ^{\otimes
        2}/\im\left(\partial_0-\partial_1+\partial_2\right)\rightarrow\Omega_{A_\QQ/\QQ},
    \end{equation*}
    and we write $D$ also for the prolongation by zero to the entire Hochschild complex.

    Since $\Tr$ is defined at the simplicial level, we look at
    \begin{equation*}
        \Tr\colon\pi_t\Brm_1\GL_1(A)\rightarrow\pi_t\Bcyc_1 A,
    \end{equation*}
    which is given by
    \begin{equation*}
        \Tr(a)=-a\otimes 1+1\otimes a
    \end{equation*}
    by Proposition~\ref{prop:homotopymap}. Applying $D$, we get
    \begin{equation*}
        D\circ\Tr(a)=d(a)
    \end{equation*}
    since $d(1)=0$.
\end{proof}

Write $\delta\colon\Sigma A\rightarrow\Sigma\Omega_{A/\QQ}$ for the suspension of $d$.

\begin{corollary}\label{cor:thhderivation}
    Let $M$ be a compact $\Hrm\QQ$-module, and let $A=\Sym_{\Hrm\QQ}M$.
    There exists a section
    \begin{equation*}
        s\colon\Sigma\Omega_{A/\QQ}\rightarrow\HH^{\Hrm\QQ}(A)
    \end{equation*}
    such that $\Tr$ factors at the level of homotopy groups as in the following commutative diagram:
    \begin{equation*}
        \xymatrix{
            \pi_{*-1}\GL_1(A)\ar[r]\ar[ddr]_{\Tr} &   \pi_{*-1} A\ar[d]^\delta\\
            &   \pi_{*-1}\Omega_{A/\QQ}\ar[d]^s\\
            &   \HH^{\Hrm\QQ}_*(A).
        }
    \end{equation*}
\end{corollary}

\begin{proof}
    We can and do assume that $M$ has zero differential, so that $A$ is a formal rational
    dga. Then, $\Omega_{A/\QQ}$ is a formal $A$-module, equivalent to $A\otimes_{\Hrm\QQ}
    M$. In particular, to give the map $s$, we just have to specify where to map generators of the
    homology of $M$. Since the map $\pi_*D\colon \HH^{\Hrm\QQ}_*(A)\rightarrow\pi_*\Sigma\Omega_{A/\QQ}$
    is surjective, such a section $s$ exists, and we simply pick one.

    Note that by our choice of $s$ and Proposition~\ref{prop:unider},
    the maps $D\circ\Tr$ and $D\circ s\circ\delta$ do agree
    at the level of homotopy groups. It follows that to prove the corollary we need to show
    that the image of $\Tr$ is contained in the image of $s$.

    Consider the B\"okstedt spectral sequence computing $\HH^{\Hrm\QQ}(A)$. There is a map
    of the homotopy colimit spectral sequences for $\BGL_1(A)$ and $\HH^{\Hrm\QQ}(A)$ (the
    B\"okstedt spectral sequence) induced by $\Tr$, which on $\Eoh_2$ terms is given by:
    \begin{equation*}
        \Hoh_s(C_*(\pi_t \Brm_\bullet\GL_1(A)))\rightarrow\Hoh_s(C_*(\pi_t\Bcyc_\bullet A)).
    \end{equation*}
    The left-hand side is concentrated in the terms $\Eoh^2_{1,t}$, and it follows
    that the image of $\Tr$ is contained in $\Eoh^{\infty}_{1,t}$. But, the proof of
    Proposition~\ref{prop:unider} implies that
    $$\Eoh^2_{1,t}=\Eoh^{\infty}_{1,t}\iso\pi_{t}\Omega_{A/\QQ}.$$
    Therefore, $\Tr$ factors through the image of $s$, as desired.
\end{proof}

It follows that the classes $\delta(x_i)$ and $\delta(y_i)$ appearing in
Corollary~\ref{cor:qhkr} can be chosen to be the suspensions of $d(x_i)$ and $d(y_i)$.

\section{The truncated Brown--Peterson spectra as algebras}\label{sec:tbps}

Fix a prime $p$. Let $\BPP$ denote the $\EE_4$-ring constructed by Basterra and
Mandell~\cite{basterra-mandell}*{Theorem~1.1} as an $\EE_4$-algebra summand of
$\MU_{(p)}$, the $p$-local complex cobordism spectrum. The homotopy ring of $\BPP$ is
\begin{equation*}
    \pi_*\BPP=\ZZ_{(p)}[v_1,v_2,\ldots],
\end{equation*}
where $v_i$ has degree $2p^i-2$. By convention, we set $v_0=p\in\pi_0\BPP$. We will work everywhere with
$\EE_1$-algebras over $\BPP$. 

\begin{lemma}
    For any $v_i$, $\BPP/v_i$ admits $\EE_1$-algebra structures over $\BPP$.
    \begin{proof}
        By Strickland~\cite{strickland}*{Corollary 3.3}, there exist products on $\BPP/v_i$.
        Moreover, by~\cite{strickland}*{Proposition 3.1}, these are homotopy associative. It follows that we
        get $A_3$ structures on $\BPP/v_i$. By
        Angeltveit~\cite{angeltveit}*{Corollary 3.7}, which Angeltveit remarks
        in the introduction to that section applies equally well over any
        $\EE_\infty$-ring, these
        extend to $A_\infty$-structures over $\MU_{(p)}$. Giving an $A_\infty$-structure over
        $\MU_{(p)}$ is equivalent to making $\BPP/v_i$ an $\EE_1$-ring over
        $\MU_{(p)}$.
        Using the $\EE_4$-ring map $\BPP\rightarrow\MU_{(p)}$
        of~\cite{basterra-mandell}*{Section 5}, we obtain $\EE_1$-algebra
        structures on $\BPP/v_i$ over $\BPP$ by restriction.
    \end{proof}
\end{lemma}
In particular, since $v_i$ is not a zero-divisor and by the lemma, the cofiber $\BPP/v_i$ has
the expected homotopy ring, namely
\begin{equation*}
    \pi_*\BPP/v_i\iso\ZZ_{(p)}[v_1,\ldots,\widehat{v_i},\ldots],
\end{equation*}
the quotient of $\pi_*\BPP$ by the ideal generated by $v_i$.

We define $\BP{n}$ as the iterated cofiber
\begin{equation*}
    \colim_{i>n}\BPP/v_{n+1}\otimes_{\BPP}\cdots\otimes_{\BPP}\BPP/v_i.
\end{equation*}

\begin{lemma}
    The truncated Brown--Peterson spectra $\BP{n}$ admit $\EE_1$-algebra structures over
    $\BPP$.
    \begin{proof}
        Since the forgetful functor $\Alg_{\BPP}\rightarrow\Mod_{\BPP}$ preserves filtered
        colimits, the underlying $\BPP$-module of the colimit
        $\colim_{i>n}\BPP/v_{n+1}\otimes_{\BPP}\cdots\otimes_{\BPP}\BPP/v_i$ in $\Alg_{\BPP}$ is $\BP{n}$.
    \end{proof}
\end{lemma}

The proof shows that for any choice of $\EE_1$-algebra structure on $\BP{n}$ and $\BPP/v_n$
over $\BPP$, we obtain an $\BPP$-algebra structure on $\BP{n-1}$.
Just as above, the homotopy rings of the truncated Brown--Peterson spectra are
\begin{equation*}
    \pi_*\BP{n}\iso\ZZ_{(p)}[v_1,\ldots,v_n]
\end{equation*}
for $n\geq 0$. Additionally, by Proposition~\ref{prop:ore}, given any $\EE_1$-algebra structure on $\BP{n}$
over $\BPP$, there is an $\EE_1$-algebra structure on $$\Erm(n)=\BP{n}[v_n^{-1}]$$ over $\BPP$
obtained by inverting $v_n$.

\begin{lemma}
    For $n\geq 1$ and any $\EE_1$-algebra structures on $\BP{n}$ and $\BPP/v_n$ over $\BPP$,
    the natural map $\BP{n}\rightarrow\BP{n-1}$ is a map of $\EE_1$-algebras over $\BPP$.
    \begin{proof}
        The map in question is the tensor product of $\BP{n}$ with the map
        $\BPP\rightarrow\BPP/v_n$ in $\Alg_{\BPP}$.
    \end{proof}
\end{lemma}

\begin{remark}
    At the moment, it is not obvious to us that the different algebra structures on $\BP{n}$ all
    result in the same $\K$-theories. Hence, we pick once and for all $\BPP$-algebra
    structures on $\BP{n-1}$ and $\BP{n}$ so that $\BP{n}\rightarrow\BP{n-1}$ is a map of
    $\BPP$-algebras. Our proofs work regardless of these choices, so there is no harm in
    them.
\end{remark}

The fiber sequence $\Sigma^{2p^n-2}\BP{n}\xrightarrow{v_n}\BP{n}\rightarrow\BP{n-1}$
exhibits $\BP{n-1}$ as a perfect right $\BP{n}$-module. The
$\EE_1$-algebra
\begin{equation*}
    \A{n-1}=\End_{\BP{n}}(\BP{n-1})^{\op}
\end{equation*}
will play a central role in this paper. Using the forgetful functor
$\Mod_{\BP{n-1}}\rightarrow\Mod_{\BP{n}}$, we get an $\EE_1$-algebra map
$\BP{n-1}\rightarrow\A{n-1}$ over $\BPP$.

\begin{lemma}\label{lem:bimodule}
    As a left $\BP{n-1}$-module, $\A{n-1}$ is equivalent to $\BP{n-1}\oplus\Sigma^{1-2p^n}\BP{n-1}$.
    \begin{proof}
        The defining sequence given by multiplication by $v_n$ on the left
        $\Sigma^{2p^n-2}\BP{n}\rightarrow\BP{n}\rightarrow\BP{n-1}$
        is a cofiber sequence of right $\BP{n}$-modules. Dualizing,
        we obtain a cofiber sequence of left $\BP{n}$-modules
        \begin{equation*}
            \Map_{\BP{n}}(\BP{n-1},\BP{n})\rightarrow\BP{n}\xrightarrow{v_n}\Sigma^{2-2p^n}\BP{n}.
        \end{equation*}
        Tensoring on the left over $\BP{n}$ with the $(\BP{n-1},\BP{n})$-bimodule $\BP{n-1}$, we obtain a fiber sequence
        \begin{equation*}
            \A{n-1}\rightarrow\BP{n-1}\xrightarrow{v_n}\Sigma^{2-2p^n}\BP{n-1}
        \end{equation*}
        of left $\BP{n-1}$-modules. It suffices now to show that $v_n$ is
        nullhomotopic as a self-map of $\BP{n-1}$. But, $v_n$ is zero in the homotopy ring
        by definition of $\BP{n-1}$. As $\BP{n-1}$ is an algebra, $v_n$ is
        nullhomotopic.
    \end{proof}
\end{lemma}

We will return to give a closer analysis of $\A{n-1}$ in the next section.

\section{Rational $\EE_\infty$-structures}\label{sec:obstruction}

The goal of this section is to show that the $\Hrm\QQ$-algebras $\BP{n-1}_\QQ$ and $\A{n-1}_\QQ$
admit $\EE_\infty$-ring structures and that $\BP{n-1}_\QQ\rightarrow\A{n-1}_\QQ$ is an
$\EE_\infty$-ring map. Our arguments for the rational $\EE_\infty$-ring structure on
$\A{n-1}_\QQ$ uses explicit rational commutative differential graded rings.

\begin{remark}
    In order to construct the $\EE_\infty$-ring structures we are interested in
    we have to replace $\BP{n}_\QQ$ and $\A{n}_\QQ$ by weakly equivalent
    models. This however does not affect the rest of the argument, as all
    functors considered in this paper are homotopy invariant.
\end{remark}

\begin{proposition}\label{prop:bpneinfinity}
    For $n\geq 0$, $\BP{n}_\QQ$ is an $\EE_\infty$-$\Hrm\QQ$-algebra.
\end{proposition}

\begin{proof}
    Fix a prime $p$. We begin by noting that $\BPP_\QQ$ admits a natural $\EE_\infty$-ring structure arising from
    $\MU_\QQ$. Indeed, recall from~\cite{adams} that $\pi_*\MU_\QQ\iso\QQ[m_1,\ldots]$, where $m_i$ is a degree $2i$
    class represented by a rational multiple of the cobordism class of $\CC\PP^i$. To construct the spectrum
    $\BPP_\QQ$ from $\MU_\QQ$, one kills each $m_i$ where $i+1$ is \emph{not} a
    $p$-power. The choice of $m_i\in\pi_{2i}\MU_\QQ$ determines by definition a map
    $$\Sym_{\Hrm\QQ}\Sigma^{2i}\Hrm\QQ\rightarrow\MU_\QQ$$ of $\EE_\infty$-algebras over
    $\Hrm\QQ$. The $\EE_\infty$-$\Hrm\QQ$-algebra $\MU_\QQ/m_i$ is the pushout
    \begin{equation*}
        \xymatrix{
            \Sym_{\Hrm\QQ}\Sigma^{2i}\Hrm\QQ\ar[d]\ar[r]    &   \MU_\QQ\ar[d]\\
            \Hrm\QQ\ar[r]                                   &   \MU_\QQ/m_i
        }
    \end{equation*}
    in $\EE_\infty$-$\Hrm\QQ$-algebras,
    and it follows from the cofiber sequence
    \begin{equation*}
        \Sigma^{2i}\Sym_{\Hrm\QQ}\Sigma^{2i}\Hrm\QQ\rightarrow\Sym_{\Hrm\QQ}\Sigma^{2i}\Hrm\QQ\rightarrow\Hrm\QQ,
    \end{equation*}
    where the right-hand map is the map of $\EE_\infty$-$\Hrm\QQ$-algebras defined by
    sending a generator to $0\in\pi_{2i}\Hrm\QQ$, that
    \begin{equation*}
        \pi_*\MU_\QQ/m_i\iso\QQ[m_1,\ldots,m_{i-1},\widehat{m}_i,m_{i+1},\ldots].
    \end{equation*}
    In other words, $\MU_\QQ/m_i$ is the cofiber of multiplication by $m_i$ on $\MU_\QQ$ \emph{as a module}.
    Just as in the previous section, we now find that $\BPP_\QQ$ has an
    $\EE_\infty$-$\Hrm\QQ$-ring
    structure obtained by taking the colimit
    \begin{equation*}
        \colim_{i\neq p^j-1}\MU_\QQ/m_i\rwe\BPP_\QQ
    \end{equation*}
    in $\CAlg_{\MU_\QQ}$, the category of $\EE_{\infty}$-algebras over $\MU_\QQ$. The same process works to reduce from $\BPP_\QQ$ to $\BP{n}_\QQ$.
\end{proof}

It would be convenient to view
$\A{n-1}_\QQ$ as an $\EE_1$-$\BP{n-1}_\QQ$-algebra now that we know that $\BP{n-1}_\QQ$ is an
$\EE_\infty$-ring. However, simply having a map of $\EE_1$-algebras $\BP{n-1}_\QQ\rightarrow
\A{n-1}_\QQ$ is not enough to guarantee this.

\begin{proposition}
    The map $\BP{n-1}_\QQ\rightarrow\A{n-1}_\QQ$ makes $\A{n-1}_\QQ$ into an $\EE_1$-$\BP{n-1}_\QQ$-algebra.
\end{proposition}

\begin{proof}
As rationalization is a localization, $\A{n-1}_\QQ$ is equivalently the endomorphism algebra of
    $\BP{n-1}_\QQ$ over $\BP{n}_\QQ$, i.e.,
    $\A{n-1}_\QQ\we\End_{\BP{n}_\QQ}(\BP{n-1}_\QQ)$. Since $\BP{n}_\QQ$ is an $\EE_\infty$-ring, $\A{n-1}_\QQ$ is
    automatically equipped with the structure of a $\BP{n}_\QQ$-algebra. Over the
    rationals, we have an $\EE_\infty$-ring map $\BP{n-1}_\QQ\rightarrow\BP{n}_\QQ$, and
    this makes $\A{n-1}_\QQ$ into a $\BP{n-1}_\QQ$-algebra. To see that this map agrees with the map
    $\BP{n-1}_\QQ\rightarrow\A{n-1}_\QQ$ coming from Morita theory note that both maps
    restrict to equivalent maps
    \begin{equation*}
        \bigoplus_{i=1}^{n-1}\Sigma^{2p^i-2}\Hrm\QQ\rightarrow\A{n-1}_\QQ,
    \end{equation*}
    which is enough to conclude since $\BP{n-1}_\QQ$ is a free $\EE_\infty$-algebra.
\end{proof}

In order to prove that $\A{n-1}_\QQ$ admits an $\EE_\infty$-ring structure,
we need to know the homotopy ring of $\A{n-1}$, which we determine now. Unfortunately, we
do not say much about the ring structure over $\SS$, but we do find $\A{n-1}_\QQ$ up to weak
equivalence.

\begin{lemma}
    The homotopy ring of $\A{n-1}$ is
    \begin{equation*}
        \pi_*\A{n-1}\iso\Lambda_{\pi_*\BP{n-1}}\langle
        \epsilon_{1-2p^n}\rangle
    \end{equation*}
    the graded exterior algebra over $\pi_*\BP{n-1}$ on one generator
    $\epsilon_{1-2p^n}$ in degree $1-2p^n$.
\end{lemma}

\begin{proof}
    Under the splitting of Lemma~\ref{lem:bimodule}, we let $\epsilon=\epsilon_{1-2p^n}$ denote the class of
    the map $\Sigma^{1-2p^n}\BP{n-1}\rightarrow\A{n-1}$. Because of degree considerations, $\epsilon^2=0$ in the
    homotopy ring of $\A{n-1}$. The description of $\pi_*\A{n-1}$ is correct
    as a left $\pi_*\BP{n-1}$-module by Lemma~\ref{lem:bimodule}. The only question is whether $v_i\epsilon=\epsilon v_i$ for $1\leq
    i\leq n-1$.
    Since $\pi_*\A{n-1}\rightarrow\pi_*\A{n-1}_\QQ$ is injective,
    it is enough to prove this statement
    rationally. By Proposition~\ref{prop:bpneinfinity}, $\BP{n}_\QQ$ is a formal rational
    dga on $v_1,\ldots,v_n$. For the
    remainder of the proof, we work entirely with $\BP{n}_\QQ$ as a formal rational dga, and
    we use dg modules over $\BP{n}_\QQ$. Let
    $X=\mathrm{cone}(v_n)$, where $v_n\colon\Sigma^{2p^n-2}\BP{n}_\QQ\rightarrow\BP{n}_\QQ$. Thus, $X$ is a dg
    $\BP{n}_\QQ$-module with
    $X_k=\pi_k\BP{n}_\QQ\oplus\pi_{k-1}\left(\Sigma^{2p^n-2}\BP{n}_\QQ\right)$, and with
    the differential $X_k\rightarrow X_{k-1}$ given by
    \begin{equation*}
        \begin{pmatrix}
            0   &   v_n\\
            0   &   0
        \end{pmatrix}.
    \end{equation*}
    Recall that $\A{n-1}_\QQ\we\End_{\BP{n}_\QQ}(X)$. 
    Each $v_i$ for $1\leq i\leq n-1$ acts as an endomorphism of $X$ in the obvious way, with
    matrix representation
     \begin{equation*}
        \begin{pmatrix}
            v_i   &   0\\
            0   &  v_i 
        \end{pmatrix}.
    \end{equation*}
    The element $\epsilon$ can be represented as well. Let
    $$\sigma\colon\pi_{*-1}\left(\Sigma^{2p^n-2}\BP{n}_\QQ\right)\rightarrow\pi_{1-2p^n+*}\BP{n}_\QQ$$ be a
    fixed $\pi_*\BP{n}_\QQ$-module isomorphism.
    Then,
    \begin{equation*}
        \epsilon=\begin{pmatrix}
            0   &   \sigma\\
            0   &   0
        \end{pmatrix}
    \end{equation*}
    from $\Sigma^{1-2p^n}X$ to $X$
    is a map of $\BP{n}_\QQ$-modules. For degree reasons, this map is
    unique up to the choice of $\sigma$, which in turn is unique up to multiplying by a
    non-zero rational number. Hence, we can assume that this $\epsilon$ represents
    $\epsilon_{1-2p^n}$ above, as $X$ is cofibrant as a $\BP{n}_\QQ$ dg module. Now, we see
    that $v_i\epsilon=\epsilon v_i$ for $1\leq i\leq n-1$.
\end{proof}

Similar ideas allow us to prove that $\A{n-1}_\QQ$ is equivalent to a commutative rational dga.

\begin{proposition}\label{prop:aneinfinity}
    The algebra $\A{n-1}_\QQ$ admits the structure of a commutative rational differential graded algebra.
\end{proposition}

\begin{proof}
    For simplicity, let $d=2p^n-2$.
    As in the proof above, we can assume that $\A{n-1}_\QQ$ is the endomorphism dga of the cone
    $X$ of $v_n$. As a graded algebra,
    \begin{equation*}
        \A{n-1}_\QQ=\bigoplus_k\Hom_{\pi_*\BP{n}_\QQ}^k(X,X),
    \end{equation*}
    where the degree $k$ part consists of homogeneous degree $k$ maps of graded $\pi_*\BP{n}_\QQ$-modules.
    Since $X$ is isomorphic to
    $\pi_*\BP{n}_\QQ\oplus\pi_*\Sigma^{d+1}\BP{n}_{\QQ}$ as a graded
    module, it follows that $\A{n-1}_\QQ$ as a graded algebra is isomorphic to the graded matrix ring
    $\Mrm_2(\pi_*\BP{n}_\QQ)$, determined by letting
    \begin{equation*}
        \begin{pmatrix}
            0   &   1\\
            0   &   0
        \end{pmatrix}
    \end{equation*}
    be in degree $-d-1=1-2p^n$. So, an element of $\A{n-1}_\QQ$ of degree $k$ is represented by an element
    \begin{equation*}
        \begin{pmatrix}
            a&b\\
            c&d
        \end{pmatrix}
    \end{equation*}
    where $a,d\in\pi_k\BP{n}_\QQ$, the element $b$ is in $\pi_{k+d+1}\BP{n}_\QQ$, and
    $c\in\pi_{k-d-1}\BP{n}_\QQ$. The differential on $\A{n-1}_\QQ$ is defined by the equation
    \begin{equation*}
        d(f)=d_X\circ f-(-1)^k f\circ d_{X}
    \end{equation*}
    if $f$ is homogeneous of degree $k$. With this convention,
    \begin{equation*}
        d\left(\begin{pmatrix}a&b\\c&d\end{pmatrix}\right)=\begin{pmatrix}v_nc&v_nd\\0&0\end{pmatrix}-(-1)^k\begin{pmatrix}0&av_n\\0
            & cv_n\end{pmatrix}=\begin{pmatrix}v_nc&v_nd-(-1)^kav_n\\0&-(-1)^kcv_n\end{pmatrix}
    \end{equation*}
    for a homogeneous element $f$ of degree $k$, written as an element of $\Mrm_2(\pi_*\BP{n}_\QQ)$.

    Since $\pi_*\BP{n}_\QQ$ is commutative and $v_n$ is a regular homogeneous element,
    the cycles of degree $k$ in $\A{n-1}_\QQ$ are all of the form
    \begin{equation*}
        \begin{pmatrix}
            a   &   b\\
            0   &   (-1)^ka
        \end{pmatrix}.
    \end{equation*}
    Let $Z$ be the graded subalgebra of $\A{n-1}_\QQ$ given by cycles such that
    $v_n$ does not divide either $a$ or $b$. This graded submodule when equipped with the trivial
    differential is a dg-submodule (over $\QQ$ or $\BP{n-1}_\QQ$),
    and the inclusion respects multiplication and is unital; that is, it is a dg-subalgebra
    of $\A{n-1}_\QQ$.
    Because of the fact that $\pi_*\BP{n}_\QQ$ is concentrated in even degrees, and as $d+1$
    is odd, $a$ and $b$ can never be simultaneously non-zero. It follows from this that $Z$
    is in fact a commutative differential graded algebra. Now, we claim that the inclusion
    $Z\rightarrow \A{n-1}_\QQ$ is a quasi-isomorphism. Indeed, the map in homology is clearly
    injective, since all boundaries in $\A{n-1}_\QQ$ contain a $v_n$-term, and it is surjective
    because all cycles in which a $v_n$-term appears must be boundaries (using again the fact that
    $a$ and $b$ in the notation above are not simultaneously non-zero). As $Z$ is a
    commutative dga, the proposition follows.
\end{proof}

\begin{remark}
    An alternative way of proving the results in this section is via Goerss--Hopkins
    obstruction theory. The obstruction groups for the existence of $\EE_\infty$-ring structures on $\A{n-1}_\QQ$ lifting the square-zero structure on coefficients vanish,
    while the groups parameterizing choices of lifts are zero except at $p=2$.
\end{remark}

\section{Rognes' question}\label{sec:arc}

In this section we give a negative answer to the question of Rognes for $n>1$ at all primes $p$.
Note that while we work with the non-$p$-completed Brown--Peterson
and Johnson--Wilson spectra for notational simplicity, no alterations are needed in the argument to handle the
$p$-complete case.

\begin{question}[Non-$p$-complete Rognes question]
    For $n>0$, is the sequence
    \begin{equation}\label{eq:arc1}
        \K(\BP{n-1})\rightarrow\K(\BP{n})\rightarrow\K(\Erm(n))
    \end{equation}
    a fiber sequence of connective spectra, where $\K(\BP{n-1})\rightarrow\K(\BP{n})$ is the transfer map?
\end{question}

An affirmative answer would
identify the fiber of a localization map in $K$-theory. Our
earlier results allow us to do this unconditionally when the localization comes from a
reasonable localization of $\EE_1$-rings. In the case of truncated Brown-Peterson spectra, we get the following result.

\begin{theorem}\label{thm:fiber1}
    Fix $n>0$, and let $\A{n-1}=\End_{\BP{n}}(\BP{n-1})^{\op}$. There is a fiber sequence
    \begin{equation*}
        \Kloc(\A{n-1})\rightarrow\Kloc(\BP{n})\rightarrow\Kloc(\Erm(n))
    \end{equation*}
    of nonconnective $\K$-theory spectra.
\end{theorem}

\begin{proof}
    Since the homotopy ring $\pi_*\BP{n}\we\ZZ_{(p)}[v_1,\ldots,v_n]$ is graded-commutative,
    the multiplicative subset generated by $v_n$ trivially
    satisfies the right Ore condition. The localization is nothing other than the
    Johnson--Wilson spectrum  $\Erm(n)$, and the cofiber of
    $\Sigma^{2p^n-2}\BP{n}\xrightarrow{v_n}\BP{n}$ is the spectrum $\BP{n-1}$,
    viewed as a right $\BP{n}$-module. The theorem follows now from Theorem~\ref{thm:kfiber}.
\end{proof}

The transfer map in the statement of the question is obtained by viewing $\BP{n-1}$ as a perfect right $\BP{n}$-module.

\begin{lemma}
    There is a commutative diagram
    \begin{equation*}
        \xymatrix{
            \Kloc(\A{n-1})\ar[r]  &   \Kloc(\BP{n})\ar[r] & \Kloc(\Erm(n))\\
            \Kloc(\BP{n-1})\ar[u]\ar[ur]&&
        }
    \end{equation*}
    where the diagonal arrow is the transfer map and the vertical map is induced from the
    algebra map $\BP{n-1}\rightarrow\A{n-1}$.
\end{lemma}

\begin{proof}
    There is a commutative diagram
    \begin{equation*}
        \xymatrix{
            \Mod_{\A{n-1}}\ar[r]    &   \Mod_{\BP{n}}\\
            \Mod_{\BP{n-1}}.\ar[u]\ar[ur]&
        }
    \end{equation*}
    Here the horizontal arrow is the fully faithful functor arising from the equivalence of
    $v_n$-nilpotent $\BP{n}$-modules and right $\A{n-1}$-modules, the diagonal arrow is
    the forgetful functor along the map of $\EE_1$-rings $\BP{n}\rightarrow\BP{n-1}$, and the
    vertical map is induced since the diagonal map lands in the subcategory of
    $v_n$-nilpotent $\BP{n}$-modules. All three functors preserve compact objects, and the
    maps in $K$-theory in the statement of the lemma are those induced by these three
    functors.
\end{proof}

The next result is a trivial consequence of the lemma, but the observation is at the heart
of our approach to the question of Rognes.

\begin{lemma}
    Suppose that~\eqref{eq:arc1} is a fiber sequence of connective spectra, then
    $\K(\BP{n-1})\rightarrow\K(\A{n-1})$ is an equivalence.
\end{lemma}

By the theorem of Blumberg and Mandell~\cite{blumberg-mandell}, when $n=1$ the
question has a positive answer, and hence the
theorem applies.
We state the analogous result for complex $K$-theory. Let $A=\End_{\ku}(\Hrm\ZZ)$. Then, the methods above give a
fiber sequence
\begin{equation*}
    \Kloc(A)\rightarrow\Kloc(\ku)\rightarrow\Kloc(\KU).
\end{equation*}
On the other hand, Blumberg and Mandell showed in~\cite{blumberg-mandell} that at the level
of connective $K$-theory, one has a fiber sequence
\begin{equation*}
    \K(\Hrm\!\ZZ)\rightarrow\K(\ku)\rightarrow\K(\KU)
\end{equation*}
of connective spectra.
It follows that $\K(\Hrm\!\ZZ)\rightarrow\K(A)$ is an equivalence.
In this case, $A$ has non-zero homotopy groups $\pi_0A\iso\ZZ$ and
$\pi_{-3}A\iso\ZZ\cdot\epsilon_{-3}$.

It is not difficult using group completion techniques to show that when $n>1$ the map
$$\Kloc_i(\BP{n-1})\rightarrow\Kloc_i(\tau_{\geq 0}\A{n-1})$$ is not an isomorphism for
general $i>0$.
However, because it seems difficult to analyze the map $\Kloc(\tau_{\geq
0}\A{n-1})\rightarrow\Kloc(\A{n-1})$,
this does not directly solve the Rognes' question. The strategy of the
main theorem will be to compute classes in the image of the trace map to topological
Hochschild homology in order to conclude that there are positive degree classes in
$\K(\A{n-1})$
not in the image of $\K(\BP{n-1})\rightarrow\K(\A{n-1})$. This of
course implies that the question of Rognes (including the $p$-complete versions) has a
negative answer for $n>1$.

Now we come to the main theorem of the paper.

\begin{theorem}
    The transfer map $\K(\BP{n-1})\rightarrow\K(\A{n-1})$ is
    not an equivalence when $n>1$. In particular, for $n>1$,
    $$\K(\BP{n-1})\rightarrow\K(\BP{n})\rightarrow\K(\Erm(n))$$ is not a fiber
    sequence of connective spectra.
\end{theorem}

\begin{proof}
    In view of the following commutative diagram, 
    \begin{equation*}
        \xymatrix{
            \K(\BP{n-1})\ar[r]\ar[d] & \HH^{\Hrm\QQ}\left(\BP{n-1}_\QQ\right)\ar[d]\\
            \K(\A{n-1})\ar[r] & \HH^{\Hrm\QQ}(\A{n-1}_\QQ),
        }
    \end{equation*}
    this is an immediate consequence of the next lemma.
\end{proof}

By Proposition~\ref{prop:aneinfinity}, $\A{n-1}_\QQ$ admits the structure of an
$\EE_\infty$-$\Hrm\QQ$-algebra. Since the homotopy ring $\pi_*\A{n-1}_\QQ$ is a
free graded-commutative $\QQ$-algebra, it follows that $\A{n-1}_\QQ$ is
equivalent to a free $\EE_\infty$-$\Hrm\QQ$-algebra, so that
Corollary~\ref{cor:qhkr} applies and gives
\begin{equation}
    \HH^{\Hrm\QQ}_*\left(\A{n-1}_\QQ\right)\iso\QQ[v_1,\ldots,v_{n-1},\delta_{2-2p^n}]\otimes\Lambda_\QQ\langle\sigma_1,\ldots,\sigma_{n-1},\epsilon_{1-2p^n}\rangle,
\end{equation}
where the degree of $\sigma_i$ is $2p^i-1$ and the degree of $\epsilon_{1-2p^n}$ is $1-2p^n$.

\begin{lemma}
    If $x=v_1^{a_1}\cdots v_{n-1}^{a_{n-1}}\epsilon_{1-2p^n}$ is a monomial in $\pi_*\A{n-1}$
    of positive total degree, i.e., $\sum_{i=1}^{n-1}a_i(p^i-1) \ge p^n$, then the class
    \begin{equation*}
        v_1^{a_1}\cdots v_{n-1}^{a_{n-1}}\delta_{2-2p^n}+\sum_{i=1}^{n-1}a_i v_1^{a_1}\cdots v_i^{a_i-1}\cdots
        v_{n-1}^{a_{n-1}}\sigma_i\epsilon_{1-2p^n}
    \end{equation*}
    is in the image of $\K(\A{n-1})\rightarrow\HH^{\Hrm\QQ}(\A{n-1}_\QQ)$ and not
    in the image of $\HH^{\Hrm\QQ}(\BP{n-1}_\QQ)\rightarrow\HH^{\Hrm\QQ}(\A{n-1}_\QQ)$.
\end{lemma}

\begin{proof}
    Consider the commutative diagram
    \begin{equation*}
        \xymatrix{
            \Sigma^\infty\BGL_1(\BP{n-1})_+\ar[r]\ar[d] &   \K(\BP{n-1})\ar[r]\ar[d] & \HH^{\Hrm\QQ}\left(\BP{n-1}_\QQ\right)\ar[d]\\
            \Sigma^\infty\BGL_1(\A{n-1})_+\ar[r]              &    \K(\A{n-1})\ar[r] &
            \HH^{\Hrm\QQ}(\A{n-1}_\QQ).
        }
    \end{equation*}
    Using Corollary~\ref{cor:qhkr}, we see that
    the right-hand vertical map is in fact an inclusion of algebras
    \begin{equation*}
        \QQ[v_1,\ldots,v_{n-1}]\otimes\Lambda_\QQ\langle\sigma_1,\ldots,\sigma_{n-1}\rangle\rightarrow\QQ[v_1,\ldots,v_{n-1},\delta_{2-2p^n}]\otimes\Lambda_\QQ\langle\sigma_1,\ldots,\sigma_{n-1},\epsilon_{1-2p^n}\rangle.
    \end{equation*}
    If $x$ has positive degree $d$, let $y$ denote the class associated to $x$ in
    $$\pi_{d+1}\BGL_1(\A{n-1})\iso\pi_{d}\A{n-1}.$$
    By Corollary~\ref{cor:thhderivation}, the class $y$ maps via the trace map
    and rationalization to the non-zero element
    \begin{equation*}
        v_1^{a_1}\cdots v_{n-1}^{a_{n-1}}\delta_{2-2p^n}+\sum_{i=1}^{n-1}a_i v_1^{a_1}\cdots v_i^{a_i-1}\cdots
        v_{n-1}^{a_{n-1}}\sigma_i\epsilon_{1-2p^n}
    \end{equation*}
    of $\HH^{\Hrm\QQ}_{d+1}(\A{n-1}_\QQ)$. Because of the
    monomials involving $\epsilon_{1-2p^n}$ and $\delta_{2-2p^n}$, this class is not in the
    image of $\HH^{\Hrm\QQ}(\BP{n-1}_\QQ)$. Of course, since its image in
    $\HH^{\Hrm\QQ}_{d+1}(\A{n-1}_\QQ)$ is non-zero, the class $y$ must map to a
    non-zero class in $\K_{d+1}(\A{n-1})$. This class cannot be in the image of
    $\K(\BP{n-1})\rightarrow\K(\A{n-1})$.
\end{proof}

The proof of the theorem requires both the negative degree class $\epsilon_{1-2p^n}$ and
the positive degree class $v_1$ in the homotopy groups of $\A{n-1}$. When $n=0,1$, there is no
$v_1$, which is why this method does not contradict the earlier theorems of Quillen and
Blumberg--Mandell.

\begin{remark}
    Recent work of Blumberg and Mandell~\cite{blumberg-mandell-thh2} provides a different approach to the
    homotopy groups of $\K(\SS)$ in the spirit of the Ausoni--Rognes program, in which analyzing
    the $K$-theory of $\Erm(n)$ is skipped in favor of looking at $\Erm_n$ directly.
    Blumberg and Mandell
    prove as an extension of their earlier work in~\cite{blumberg-mandell} that there is a
    fiber sequence
    \begin{equation}\label{eq:four}
        \K(\WW[\![u_1,\ldots,u_{n-1}]\!])\rightarrow\K(\BPP_n)\rightarrow\K(\Erm_n)
    \end{equation}
    of connective spectra
    for all $n>0$, where $\BPP_n$ denotes the connective cover of $\Erm_n$, $\WW$ is the $p$-typical Witt ring, and the $u_i$ are in degree $0$.
    Note, however, that the Ausoni--Rognes program in principle allows a computation of $\K(\BP{n})$,
    whereas $\K(\BPP_n)$ is more difficult to compute using their techniques.
\end{remark}

In the end, the successful d\'evissage-type results of Quillen~\cite{quillen},
Blumberg--Mandell~\cites{blumberg-mandell,blumberg-mandell-thh2}, and
Barwick--Lawson~\cite{barwick-lawson} can all be expressed in terms of Barwick's
theorem of the heart~\cite{barwick-heart}*{Theorem~8.7}. For example, consider
the $\EE_\infty$-ring spectrum $\ku$ and its localization $\KU=\ku[\beta^{-1}]$.
In the notation of Section~\ref{sub:loce1}, there is an exact sequence
$\Mod_{\ku}^{\mathrm{Nil}(S),\omega}\rightarrow\Mod_\ku^\omega\rightarrow\Mod_{\KU}^\omega$,
where $S=\{1,\beta,\beta^2,\ldots\}$. The natural, Postnikov $t$-structure on
$\Mod_\ku$ is \emph{not} bounded, but it restricts to a bounded $t$-structure
on $\Mod_{\ku}^{\mathrm{Nil}(S),\omega}$ with heart the category of finitely
generated abelian groups.

Barwick's theorem of the heart says that the connective $K$-theory of a stable
$\infty$-category with a bounded $t$-structure is equivalent to the connective
$K$-theory of the heart. Thus, we obtain the fiber sequence
$\K(\ZZ)\rightarrow\K(\ku)\rightarrow\K(\KU)$ of connective spectra. 
The same argument establishes the fiber sequences in~\eqref{eq:four}.
In fact, Barwick's theorem was recently extended
in~\cite{antieau-gepner-heller} to negative $K$-theory when
the heart is noetherian. Thus, there are fiber sequences
$$\Kloc(\ZZ)\rightarrow\Kloc(\ku)\rightarrow\Kloc(\KU)$$ and
$$\Kloc(\WW[\![u_1,\ldots,u_{n-1}]\!])\rightarrow\Kloc(\BPP_n)\rightarrow\Kloc(\Erm_n)$$
of \emph{nonconnective} $K$-theory spectra. This amounts to proving
that the negative $K$-theory of $\KU$ and $\Erm_n$ vanish. On the other hand,
the following problem is open.

\begin{question}
    Do the negative $K$-groups of $\Erm(n)$ vanish for $n\geq 2$?
\end{question}

Since $\pi_0\BP{n}\iso\ZZ_{(p)}$ is regular and noetherian, $\K_{-i}(\BP{n})=0$
for $i\geq 1$ by~\cite{bgt1}*{Theorem~9.53}. So, we may ask well as about the negative $K$-theory of
$\A{n-1}$. The question of Rognes can be viewed as asking for some
structure on $\Mod_{\A{n-1}}^\omega$ that generalizes that of a $t$-structure,
with generalized heart equivalent to $\Mod_{\BP{n-1}}^\omega$,
as well as for a generalization of the theorem of the heart. The main theorem
of this paper shows that this is too much to hope for.

\vspace{20pt}
\noindent
\textsc{University of Illinois at Chicago, USA}\\
\emph{E-mail address:} \texttt{antieau@math.uic.edu}

\vspace{10pt}
\noindent
\textsc{University of Copenhagen, Denmark}\\
\emph{E-mail address:} \texttt{tbarthel@math.ku.dk}

\vspace{10pt}
\noindent
\textsc{Purdue University, USA}\\
\emph{E-mail address:} \texttt{dgepner@purdue.edu}


\begin{bibdiv}
\begin{biblist}

% \bib{abramovich-olsson-vistoli}{article}{
%     author={Abramovich, Dan},
%     author={Olsson, Martin},
%     author={Vistoli, Angelo},
%     title={Tame stacks in positive characteristic},
%     journal={Ann. Inst. Fourier (Grenoble)},
%     volume={58},
%     date={2008},
%     number={4},
%     pages={1057--1091},
%     issn={0373-0956},
% %     review={\MR{2427954 (2009c:14002)}},
% }
% % 
% % \bib{abramovich-vistoli}{article}{
% %     author={Abramovich, Dan},
% %     author={Vistoli, Angelo},
% %     title={Compactifying the space of stable maps},
% %     journal={J. Amer. Math. Soc.},
% %     volume={15},
% %     date={2002},
% %     number={1},
% %     pages={27--75},
% %     issn={0894-0347},
% % %     review={\MR{1862797 (2002i:14030)}},
% % %     doi={10.1090/S0894-0347-01-00380-0},
% % }

\bib{adams}{book}{
    author={Adams, J. F.},
    title={Stable homotopy and generalised homology},
    note={Chicago Lectures in Mathematics},
    publisher={University of Chicago Press, Chicago, Ill.-London},
    date={1974},
    pages={x+373},
%     review={\MR{0402720 (53 \#6534)}},
}
% 
% \bib{atjlss}{article}{
%     author={Alonso Tarr{\'{\i}}o, Leovigildo},
%     author={Jerem{\'{\i}}as L{\'o}pez, Ana},
%     author={Souto Salorio, Mar{\'{\i}}a Jos{\'e}},
%     title={Bousfield localization on formal schemes},
%     journal={J. Algebra},
%     volume={278},
%     date={2004},
%     number={2},
%     pages={585--610},
%     issn={0021-8693},
% %     review={\MR{2071654 (2005g:14037)}},
% %     doi={10.1016/j.jalgebra.2004.02.030},
% }
%     
% % \bib{abg}{article}{
% %     author = {Ando, Matthew},
% %     author = {Blumberg, Andrew J.},
% %     author = {Gepner, David},
% %     title = {Parameterized homotopy theory and twisted Umkehr maps},
% %     journal = {ArXiv e-prints},
% %     eprint = {http://arxiv.org/abs/1112.2203},
% %     year = {2011},
% % }
% 
\bib{angeltveit}{article}{
    author={Angeltveit, Vigleik},
    title={Topological Hochschild homology and cohomology of $A_\infty$
    ring spectra},
    journal={Geom. Topol.},
    volume={12},
    date={2008},
    number={2},
    pages={987--1032},
    issn={1465-3060},
%     review={\MR{2403804 (2009e:55012)}},
%     doi={10.2140/gt.2008.12.987},
}
% % \bib{antieau-thesis}{thesis}{
% %     author = {Antieau, Benjamin},
% %     title = {The spectral index of Brauer classes},
% %     note = {Ph.D. thesis, UIC (2010), available at http://www.math.ucla.edu/\textasciitilde antieau/},
% % }
% 
\bib{ag}{article}{
    author = {Antieau, Benjamin},
    author = {Gepner, David},
    title = {Brauer groups and \'etale cohomology in derived algebraic geometry},
    journal = {Geom. Topol.},
    volume = {18},
    year = {2014},
    number = {2},
    pages = {1149--1244},
}
% 
% % \bib{agg}{article}{
% %     author = {Antieau, Benjamin},
% %     author = {Gepner, David},
% %     author = {G\'omez, Jos\'e Manuel},
% %     title = {Actions of Eilenberg-MacLane spaces on K-theory spectra and uniqueness of twisted K-theory},
% %     note = {To appear in Trans. Amer. Math. Soc.},
% % }

\bib{antieau-gepner-heller}{article}{
    author = {Antieau, Benjamin},
    author = {Gepner, David},
    author = {Heller, Jeremiah},
    title = {On the theorem of the heart in negative $K$-theory},
    journal = {ArXiv e-prints},
    eprint = {http://arxiv.org/abs/1610.07207},
    year = {2016},
}
% 
% % \bib{arinkin-gaitsgory}{article}{
% %     author = {Arinkin, D.},
% %     author = {Gaitsgory, D.},
% %     title = {Singular support of coherent sheaves, and the geometric Langlands conjecture},
% %   journal = {ArXiv e-prints},
% %      eprint = {http://arxiv.org/abs/1201.6343},
% %      year = {2012},
% % }
% 
% % \bib{artin}{article}{
% %     author={Artin, M.},
% %     title={Versal deformations and algebraic stacks},
% %     journal={Invent. Math.},
% %     volume={27},
% %     date={1974},
% %     pages={165--189},
% %     issn={0020-9910},
% % %     review={\MR{0399094 (53 \#2945)}},
% % }
% 
% % \bib{artin-mumford}{article}{
% %     author={Artin, M.},
% %     author={Mumford, D.},
% %     title={Some elementary examples of unirational varieties which are not
% %     rational},
% %     journal={Proc. London Math. Soc. (3)},
% %     volume={25},
% %     date={1972},
% %     pages={75--95},
% %     issn={0024-6115},
% % %     review={\MR{0321934 (48 \#299)}},
% % }
% 
% % \bib{auslander-goldman}{article}{
% % author={Auslander, Maurice},
% % author={Goldman, Oscar},
% % title={The Brauer group of a commutative ring},
% % journal={Trans. Amer. Math. Soc.},
% % volume={97},
% % date={1960},
% % pages={367--409},
% % issn={0002-9947},
% % % review={\MR{0121392 (22 \#12130)}},
% % }

\bib{ausoni-rognes}{article}{
    author={Ausoni, Christian},
    author={Rognes, John},
    title={Algebraic $K$-theory of topological $K$-theory},
    journal={Acta Math.},
    volume={188},
    date={2002},
    number={1},
    pages={1--39},
    issn={0001-5962},
%     review={\MR{1947457 (2004f:19007)}},
%     doi={10.1007/BF02392794},
}
% 
% % \bib{azumaya}{article}{
% %     author={Azumaya, Gor{\^o}},
% %     title={On maximally central algebras},
% %     journal={Nagoya Math. J.},
% %     volume={2},
% %     date={1951},
% %     pages={119--150},
% %     issn={0027-7630},
% % %     review={\MR{0040287 (12,669g)}},
% % }
% 
% % \bib{baker-lazarev}{article}{
% %     author={Baker, Andrew},
% %     author={Lazarev, Andrey},
% %     title={Topological Hochschild cohomology and generalized Morita
% %     equivalence},
% %     journal={Algebr. Geom. Topol.},
% %     volume={4},
% %     date={2004},
% %     pages={623--645},
% %     issn={1472-2747},
% % %     review={\MR{2100675 (2005g:55009)}},
% % %     doi={10.2140/agt.2004.4.623},
% % }
% 
% % \bib{baker-richter}{article}{
% %       author={Baker, Andrew},
% %       author={Richter, Birgit},
% %        title={Invertible modules for commutative {$\mathds S$}-algebras with residue fields},
% %         date={2005},
% %         ISSN={0025-2611},
% %      journal={Manuscripta Math.},
% %       volume={118},
% %       number={1},
% %        pages={99\ndash 119},
% % %          url={http://dx.doi.org/10.1007/s00229-005-0582-1},
% % %       review={\MR{MR2171294}},
% % }
% 
% % \bib{baker-richter-szymik}{article}{
% % author = {Baker, A.},
% % author = {Richter, B.},
% % author = {Szymik, M.},
% % title = {Brauer groups for commutative {$\mathds S$}-algebras},
% % journal = {ArXiv e-prints},
% % eprint = {http://arxiv.org/abs/1005.5370},
% % year = {2010},
% % }
% 
% \bib{balmer}{article}{
%     author={Balmer, Paul},
%     title={The spectrum of prime ideals in tensor triangulated categories},
%     journal={J. Reine Angew. Math.},
%     volume={588},
%     date={2005},
%     pages={149--168},
%     issn={0075-4102},
% %     review={\MR{2196732 (2007b:18012)}},
% %     doi={10.1515/crll.2005.2005.588.149},
% }
% 
% \bib{balmer-favi}{article}{
%     author={Balmer, Paul},
%     author={Favi, Giordano},
%     title={Generalized tensor idempotents and the telescope conjecture},
%     journal={Proc. Lond. Math. Soc. (3)},
%     volume={102},
%     date={2011},
%     number={6},
%     pages={1161--1185},
%     issn={0024-6115},
% %     review={\MR{2806103 (2012d:18010)}},
% %     doi={10.1112/plms/pdq050},
% }

\bib{barwick-q}{article}{
    author = {Barwick, Clark},
    title = {On the Q construction for exact {$\infty$}-categories},
    journal = {ArXiv e-prints},
    eprint = {http://arxiv.org/abs/1301.4725},
    year = {2013},
}

\bib{barwick-heart}{article}{
   author={Barwick, Clark},
   title={On exact $\infty$-categories and the theorem of the heart},
   journal={Compos. Math.},
   volume={151},
   date={2015},
   number={11},
   pages={2160--2186},
   issn={0010-437X},
%    review={\MR{3427577}},
%    doi={10.1112/S0010437X15007447},
}

\bib{barwick-highercats}{article}{
   author={Barwick, Clark},
   title={On the algebraic $K$-theory of higher categories},
   journal={J. Topol.},
   volume={9},
   date={2016},
   number={1},
   pages={245--347},
   issn={1753-8416},
%    review={\MR{3465850}},
%    doi={10.1112/jtopol/jtv042},
}

\bib{barwick-lawson}{article}{
    author = {Barwick, Clark},
    author = {Lawson, Tyler},
    title = {Regularity of structured ring spectra and localization in K-theory},
    journal = {ArXiv e-prints},
    eprint = {http://arxiv.org/abs/1402.6038},
    year = {2014},
}

\bib{barwick-rognes}{article}{
    author = {Barwick, Clark},
    author = {Rognes, John},
    title = {On the Q construction for exact {$\infty$}-categories},
    eprint = {http://dl.dropbox.com/u/1741495/papers/qconstr.pdf},
    note = {Accessed 19 June 2017},
    year = {2013},
}

\bib{basterra-mandell}{article}{
    author={Basterra, Maria},
    author={Mandell, Michael A.},
    title={The multiplication on BP},
    journal={J. Topol.},
    volume={6},
    date={2013},
    number={2},
    pages={285--310},
    issn={1753-8416},
%     review={\MR{3065177}},
%     doi={10.1112/jtopol/jts032},
}
% 
% % \bib{bazzoni}{article}{
% %     author = {Bazzoni, Silvana},
% %     title = {The telescope conjecture for semihereditary commutative rings},
% %     journal = {ArXiv e-prints},
% %     eprint = {http://arxiv.org/abs/1301.2553},
% %     year = {2013},
% % }
% 
% % \bib{beilinson}{article}{
% %     author={Be{\u\i}linson, A. A.},
% %     title={Coherent sheaves on ${\bf P}^{n}$ and problems in linear
% %     algebra},
% %     journal={Funktsional. Anal. i Prilozhen.},
% %     volume={12},
% %     date={1978},
% %     number={3},
% %     pages={68--69},
% %     issn={0374-1990},
% % %     review={\MR{509388 (80c:14010b)}},
% % }
% 
% \bib{benson-iyengar-krause}{article}{
%     author={Benson, David J.},
%     author={Iyengar, Srikanth B.},
%     author={Krause, Henning},
%     title={Stratifying modular representations of finite groups},
%     journal={Ann. of Math. (2)},
%     volume={174},
%     date={2011},
%     number={3},
%     pages={1643--1684},
%     issn={0003-486X},
% %     review={\MR{2846489}},
% %     doi={10.4007/annals.2011.174.3.6},
% }
% 
% % \bib{bzfn}{article}{
% %     author={Ben-Zvi, David},
% %     author={Francis, John},
% %     author={Nadler, David},
% %     title={Integral transforms and Drinfeld centers in derived algebraic
% %     geometry},
% %     journal={J. Amer. Math. Soc.},
% %     volume={23},
% %     date={2010},
% %     number={4},
% %     pages={909--966},
% %     issn={0894-0347},
% % %     review={\MR{2669705 (2011j:14023)}},
% % %     doi={10.1090/S0894-0347-10-00669-7},
% % }
% 
% % \bib{bernardara}{article}{
% %     author={Bernardara, Marcello},
% %     title={A semiorthogonal decomposition for Brauer-Severi schemes},
% %     journal={Math. Nachr.},
% %     volume={282},
% %     date={2009},
% %     number={10},
% %     pages={1406--1413},
% %     issn={0025-584X},
% % %     review={\MR{2571702 (2010k:14021)}},
% % %     doi={10.1002/mana.200610826},
% % }
% 
% \bib{blumberg-cohen-schlichtkrull}{article}{
%     author={Blumberg, Andrew J.},
%     author={Cohen, Ralph L.},
%     author={Schlichtkrull,
%               Christian},
%     title={Topological {H}ochschild homology of {T}hom spectra and the
%               free loop space},
%     journal={Geom. Topol.},
%     volume={14},
%     date={2010},
%     number={2},
%     pages={1165--1242},
%     issn={1465-3060},
% }
% % \bib{bgt1}{article}{
% %     author = {{Blumberg}, A.~J.},
% %     author = {{Gepner}, D.},
% %     author = {{Tabuada}, G.},
% %     title    =    {A    universal    characterization    of    higher    algebraic    K-theory},
% %     journal = {ArXiv e-prints},
% %     eprint = {http://arxiv.org/abs/1001.2282},
% %     year = {2010},
% % }
% 
\bib{bgt1}{article}{
    author={Blumberg, Andrew J.},
    author={Gepner, David},
    author={Tabuada, Gon{\c{c}}alo},
    title={A universal characterization of higher algebraic $K$-theory},
    journal={Geom. Topol.},
    volume={17},
    date={2013},
    number={2},
    pages={733--838},
    issn={1465-3060},
%     review={\MR{3070515}},
%     doi={10.2140/gt.2013.17.733},
}

\bib{bgt-endomorphisms}{article}{
   author={Blumberg, Andrew J.},
   author={Gepner, David},
   author={Tabuada, Gon\c{c}alo},
   title={$K$-theory of endomorphisms via noncommutative motives},
   journal={Trans. Amer. Math. Soc.},
   volume={368},
   date={2016},
   number={2},
   pages={1435--1465},
   issn={0002-9947},
%    review={\MR{3430369}},
%    doi={10.1090/tran/6507},
}

%  
% % \bib{bgt2}{article}{
% %     author = {{Blumberg}, A.~J.},
% %     author = {{Gepner}, D.},
% %     author = {{Tabuada}, G.},
% %     title    =    {Uniqueness of the multiplicative cyclotomic trace},
% %     journal = {ArXiv e-prints},
% %     eprint = {http://arxiv.org/abs/1103.3923},
% %     year = {2011},
% % }

\bib{blumberg-mandell}{article}{
    author={Blumberg, Andrew J.},
    author={Mandell, Michael A.},
    title={The localization sequence for the algebraic $K$-theory of topological $K$-theory},
    journal={Acta Math.},
    volume={200},
    date={2008},
    number={2},
    pages={155--179},
    issn={0001-5962},
%     review={\MR{2413133 (2009f:19003)}},
%     doi={10.1007/s11511-008-0025-4},
}

\bib{blumberg-mandell-thh2}{article}{
    author={Blumberg, Andrew J.},
    author={Mandell, Michael A.},
    title = {Localization for $THH(ku)$ and the topological Hochschild and cyclic homology of Waldhausen
        categories},
  journal = {ArXiv e-prints},
     eprint = {http://arxiv.org/abs/1111.4003},
     year = {2011},
}

\bib{blumberg-mandell-thh}{article}{
    author={Blumberg, Andrew J.},
    author={Mandell, Michael A.},
    title={Localization theorems in topological Hochschild homology and
    topological cyclic homology},
    journal={Geom. Topol.},
    volume={16},
    date={2012},
    number={2},
    pages={1053--1120},
    issn={1465-3060},
%     review={\MR{2928988}},
%     doi={10.2140/gt.2012.16.1053},
}

\bib{blumberg-mandell-ksphere}{article}{
    author={Blumberg, Andrew J.},
    author={Mandell, Michael A.},
    title = {The homotopy groups of the algebraic K-theory of the sphere spectrum},
  journal = {ArXiv e-prints},
     eprint = {http://arxiv.org/abs/1408.0133},
     year = {2014},
}

\bib{blumberg-mandell-tp}{article}{
    author={Blumberg, Andrew J.},
    author={Mandell, Michael A.},
    title = {Tate-Poitou duality and the fiber of the cyclotomic trace for the
    sphere spectrum},
  journal = {ArXiv e-prints},
     eprint = {http://arxiv.org/abs/1508.00014},
     year = {2015},
}


\bib{bokstedt-hsiang-madsen}{article}{
    author={B{\"o}kstedt, M.},
    author={Hsiang, W. C.},
    author={Madsen, I.},
    title={The cyclotomic trace and algebraic $K$-theory of spaces},
    journal={Invent. Math.},
    volume={111},
    date={1993},
    number={3},
    pages={465--539},
    issn={0020-9910},
%     review={\MR{1202133 (94g:55011)}},
%     doi={10.1007/BF01231296},
}
% 
\bib{bokstedt-neeman}{article}{
    author={B{\"o}kstedt, Marcel},
    author={Neeman, Amnon},
    title={Homotopy limits in triangulated categories},
    journal={Compos. Math.},
    volume={86},
    date={1993},
    number={2},
    pages={209--234},
    issn={0010-437X},
%     review={\MR{1214458 (94f:18008)}},
}
% 
% % \bib{bondal}{article}{
% %     author={Bondal, A. I.},
% %     title={Representations of associative algebras and coherent sheaves},
% %     journal={Izv. Akad. Nauk},
% %     volume={53},
% %     date={1989},
% %     number={1},
% %     pages={25--44},
% %     issn={0373-2436},
% %     translation={
% %         journal={Math. USSR-Izv.},
% %         volume={34},
% %         date={1990},
% %         number={1},
% %         pages={23--42},
% %         issn={0025-5726},
% %     },
% % %     review={\MR{992977 (90i:14017)}},
% % }
% 
% \bib{bondal-vandenbergh}{article}{
%     author={Bondal, A.},
%     author={van den Bergh, M.},
%     title={Generators and representability of functors in commutative and noncommutative geometry},
%     journal={Mosc. Math. J.},
%     volume={3},
%     date={2003},
%     number={1},
%     pages={1--36, 258},
%     issn={1609-3321},
% %     review={\MR{1996800 (2004h:18009)}},
% }
% 
% % \bib{bondal-kapranov}{article}{
% %     author={Bondal, A. I.},
% %     author={Kapranov, M. M.},
% %     title={Representable functors, Serre functors, and reconstructions},
% %     journal={Izv. Akad. Nauk},
% %     volume={53},
% %     date={1989},
% %     number={6},
% %     pages={1183--1205, 1337},
% %     issn={0373-2436},
% %     translation={
% %         journal={Math. USSR-Izv.},
% %         volume={35},
% %         date={1990},
% %         number={3},
% %         pages={519--541},
% %         issn={0025-5726},
% %     },
% % %     review={\MR{1039961 (91b:14013)}},
% % }
% 
% % \bib{bondal-orlov}{article}{
% %     author={Bondal, Alexei},
% %     author={Orlov, Dmitri},
% %     title={Reconstruction of a variety from the derived category and groups
% %     of autoequivalences},
% %     journal={Compos. Math.},
% %     volume={125},
% %     date={2001},
% %     number={3},
% %     pages={327--344},
% %     issn={0010-437X},
% % %     review={\MR{1818984 (2001m:18014)}},
% % %     doi={10.1023/A:1002470302976},
% % }
%  
% % \bib{borceux-vitale}{article}{
% %     author={Borceux, Francis},
% %     author={Vitale, Enrico},
% %     title={Azumaya categories},
% %     journal={Appl. Categ. Structures},
% %     volume={10},
% %     date={2002},
% %     number={5},
% %     pages={449--467},
% %     issn={0927-2852},
% % %     review={\MR{1937232 (2003h:18007)}},
% % %     doi={10.1023/A:1020570213428},
% % }
%  
% \bib{bousfield-kan}{book}{
%     author={Bousfield, A. K.},
%     author={Kan, D.  M.},
%     title={Homotopy limits, completions and localizations},
%     series={Lecture Notes in Mathematics, Vol. 304},
%     publisher={Springer-Verlag},
%     place={Berlin},
%     date={1972},
%     pages={v+348},
% %     review={\MR{0365573 (51 \#1825)}},
% }
% 
% \bib{bousfield-localization}{article}{
%     author={Bousfield, A. K.},
%     title={The localization of spectra with respect to homology},
%     journal={Topology},
%     volume={18},
%     date={1979},
%     number={4},
%     pages={257--281},
%     issn={0040-9383},
% %     review={\MR{551009 (80m:55006)}},
% %     doi={10.1016/0040-9383(79)90018-1},
% }

% \bib{bousfield}{article}{
%     author={Bousfield, A. K.},
%     title={Homotopy spectral sequences and obstructions},
%     journal={Israel J. Math.},
%     volume={66},
%     date={1989},
%     number={1-3},
%     pages={54--104},
%     issn={0021-2172},
% %     review={\MR{1017155 (91a:55027)}},
% %     doi={10.1007/BF02765886},
% }
% 
% \bib{bridgeland-king-reid}{article}{
%     author={Bridgeland, Tom},
%     author={King, Alastair},
%     author={Reid, Miles},
%     title={The McKay correspondence as an equivalence of derived categories},
%     journal={J. Amer. Math. Soc.},
%     volume={14},
%     date={2001},
%     number={3},
%     pages={535--554},
%     issn={0894-0347},
% %     review={\MR{1824990 (2002f:14023)}},
% %     doi={10.1090/S0894-0347-01-00368-X},
% }
%  
% % \bib{brown-gersten}{article}{
% %     author={Brown, Kenneth S.},
% %     author={Gersten, Stephen M.},
% %     title={Algebraic $K$-theory as generalized sheaf cohomology},
% %     conference={title={Algebraic K-theory, I: Higher K-theories (Proc. Conf., Battelle Memorial Inst., Seattle, Wash., 1972)}, },
% %     book={
% %         publisher={Springer},
% %         place={Berlin},
% %     },
% %     date={1973},
% %     pages={266--292. Lecture Notes in Math., Vol.  341},
% % %     review={\MR{0347943 (50 \#442)}},
% % }
% 
% \bib{bruning}{article}{
%     author={Br{\"u}ning, Kristian},
%     title={Thick subcategories of the derived category of a hereditary
%     algebra},
%     journal={Homology, Homotopy Appl.},
%     volume={9},
%     date={2007},
%     number={2},
%     pages={165--176},
%     issn={1532-0073},
% %     review={\MR{2366948 (2009d:18018)}},
% }
% 
% % \bib{caldararu}{thesis}{
% %     author = {C\u{a}ld\u{a}raru, Andrei},
% %     title = {Derived categories of twisted sheaves on Calabi-Yau manifolds},
% %     note = {Ph.D. thesis, Cornell University (2000), available at
% %     http://www.math.wisc.edu/\textasciitilde andreic/},
% % }
% 
% % \bib{cisinski-tabuada}{article}{
% %     author = {Cisinski, Denis-Charles},
% %     author = {Tabuada, Gonçalo},
% %     title = {Symmetric monoidal structure on non-commutative motives},
% %     journal = {Journal of K-Theory},
% %     volume = {9},
% %     number = {02},
% %     pages = {201-268},
% %     year = {2012},
% % %     doi = {10.1017/is011011005jkt169},
% % %     URL = {http://dx.doi.org/10.1017/S1865243309998841},
% % %     eprint = {http://journals.cambridge.org/article_S1865243309998841},
% % }
% 
% % \bib{dejong}{unpublished}{
% %     author={de Jong, Aise Johan},
% %     title={A result of Gabber},
% %     note={available at http://www.math.columbia.edu/~dejong/},
% % }
% 
% \bib{dellambrogio-stevenson}{article}{
%     author={Dell'Ambrogio, Ivo},
%     author={Stevenson, Greg},
%     title={On the derived category of a graded commutative Noetherian ring},
%     journal={J. Algebra},
%     volume={373},
%     date={2013},
%     pages={356--376},
%     issn={0021-8693},
% %     review={\MR{2995031}},
% %     doi={10.1016/j.jalgebra.2012.09.038},
% }
% 
% 
% % \bib{demeyer}{article}{
% %     author={DeMeyer, F. R.},
% %     title={The Brauer group of affine curves},
% %     conference={
% %         title={Brauer groups (Proc. Conf., Northwestern Univ., Evanston, Ill., 1975)},
% %     },
% %     book={
% %         publisher={Springer},
% %         place={Berlin},
% %     },
% %     date={1976},
% %     pages={16--24. Lecture Notes in Math., Vol.
% %     549},
% % %     review={\MR{0437537 (55 \#10461)}},
% % }
% 
% % \bib{demeyer-ingraham}{book}{
% %       author={DeMeyer, Frank R.},
% %       author={Ingraham, Edward},
% %        title={Separable algebras over commutative rings},
% %       series={Lecture Notes in Mathematics, Vol. 181},
% %    publisher={Springer-Verlag},
% %      address={Berlin},
% %         date={1971},
% % %       review={\MR{MR0280479}},
% % }
% 
% % \bib{deninger}{article}{
% %     author={Deninger, C.},
% %     title={A proper base change theorem for nontorsion sheaves in \'etale
% %     cohomology},
% %     journal={J. Pure Appl. Algebra},
% %     volume={50},
% %     date={1988},
% %     number={3},
% %     pages={231--235},
% %     issn={0022-4049},
% % %     review={\MR{938616 (89f:14016)}},
% % %     doi={10.1016/0022-4049(88)90102-8},
% % }
% 
% % \bib{dhs}{article}{
% %     author={Devinatz, Ethan S.},
% %     author={Hopkins, Michael J.},
% %     author={Smith, Jeffrey H.},
% %     title={Nilpotence and stable homotopy theory. I},
% %     journal={Ann. of Math. (2)},
% %     volume={128},
% %     date={1988},
% %     number={2},
% %     pages={207--241},
% %     issn={0003-486X},
% % %     review={\MR{960945 (89m:55009)}},
% % %     doi={10.2307/1971440},
% % }
% 
% \bib{devinatz-hopkins2}{article}{
%    author={Devinatz, Ethan S.},
%    author={Hopkins, Michael J.},
%    title={Homotopy fixed point spectra for closed subgroups of the Morava
%    stabilizer groups},
%    journal={Topology},
%    volume={43},
%    date={2004},
%    number={1},
%    pages={1--47},
%    issn={0040-9383},
% %    review={\MR{2030586}},
% %    doi={10.1016/S0040-9383(03)00029-6},
% }
% \bib{davis_hfp}{article}{
%    author={Davis, Daniel G.},
%    title={Homotopy fixed points for $L_{K(n)}(E_n\wedge X)$ using the
%    continuous action},
%    journal={J. Pure Appl. Algebra},
%    volume={206},
%    date={2006},
%    number={3},
%    pages={322--354},
%    issn={0022-4049},
% %    review={\MR{2235364}},
% %    doi={10.1016/j.jpaa.2005.06.022},
% }
% \bib{dubey-mallick}{article}{
%     author={Dubey, Umesh V.},
%     author={Mallick, Vivek M.},
%     title={Spectrum of some triangulated categories},
%     journal={J. Algebra},
%     volume={364},
%     date={2012},
%     pages={90--118},
%     issn={0021-8693},
% %     review={\MR{2927050}},
% %     doi={10.1016/j.jalgebra.2012.04.030},
% }

% \bib{dugger-shipley-postnikov}{article}{
%     author={Dugger, Daniel},
%     author={Shipley, Brooke},
%     title={Postnikov extensions of ring spectra},
%     journal={Algebr. Geom. Topol.},
%     volume={6},
%     date={2006},
%     pages={1785--1829},
%     issn={1472-2747},
% %     review={\MR{2263050 (2007g:55007)}},
% %     doi={10.2140/agt.2006.6.1785},
% }

% \bib{dundas-mccarthy}{article}{
%     author={Dundas, Bj{\o}rn Ian},
%     author={McCarthy, Randy},
%     title={Stable $K$-theory and topological Hochschild homology},
%     journal={Ann. of Math. (2)},
%     volume={140},
%     date={1994},
%     number={3},
%     pages={685--701},
%     issn={0003-486X},
%     review={\MR{1307900 (96e:19005a)}},
%     doi={10.2307/2118621},
% }

\bib{dundas-mccarthy-hochschild}{article}{
    author={Dundas, Bj{\o}rn Ian},
    author={McCarthy, Randy},
    title={Topological Hochschild homology of ring functors and exact
    categories},
    journal={J. Pure Appl. Algebra},
    volume={109},
    date={1996},
    number={3},
    pages={231--294},
    issn={0022-4049},
%     review={\MR{1388700 (97i:19001)}},
%     doi={10.1016/0022-4049(95)00089-5},
}
% 
% % \bib{duskin}{article}{
% %     author={Duskin, John W.},
% %     title={The Azumaya complex of a commutative ring},
% %     conference={
% %     title={Categorical algebra and its applications},
% %     address={Louvain-La-Neuve},
% %     date={1987},
% %     },
% %     book={
% %     series={Lecture Notes in Math.},
% %     volume={1348},
% %     publisher={Springer},
% %     place={Berlin},
% %     },
% %     date={1988},
% %     pages={107--117},
% % %     review={\MR{975963 (90c:13003)}},
% % %     doi={10.1007/BFb0081352},
% % }
% 
% \bib{dwyer-palmieri}{article}{
%     author={Dwyer, W. G.},
%     author={Palmieri, J. H.},
%     title={The Bousfield lattice for truncated polynomial algebras},
%     journal={Homology Homotopy Appl.},
%     volume={10},
%     date={2008},
%     number={1},
%     pages={413--436},
%     issn={1532-0073},
% %     review={\MR{2426110 (2009e:18021)}},
% }
% 
% % \bib{ehkv}{article}{
% %     author={Edidin, Dan},
% %     author={Hassett, Brendan},
% %     author={Kresch, Andrew},
% %     author={Vistoli, Angelo},
% %     title={Brauer groups and quotient stacks},
% %     journal={Amer. J. Math.},
% %     volume={123},
% %     date={2001},
% %     number={4},
% %     pages={761--777},
% %     issn={0002-9327},
% % %     review={\MR{1844577 (2002f:14002)}},
% % }
%  
\bib{ekmm}{book}{
    author={Elmendorf, A. D.},
    author={Kriz, I.},
    author={Mandell, M. A.},
    author={May, J. P.},
    title={Rings, modules, and algebras in stable homotopy theory},
    series={Mathematical Surveys and Monographs},
    volume={47},
    note={With an appendix by M. Cole},
    publisher={American Mathematical Society},
    place={Providence, RI},
    date={1997},
    pages={xii+249},
    isbn={0-8218-0638-6},
    % review={\MR{1417719 (97h:55006)}},
}
% 
% 
% % \bib{fga}{book}{
% %       author={Fantechi, Barbara},
% %       author={G{\"o}ttsche, Lothar},
% %       author={Illusie, Luc},
% %       author={Kleiman, Steven~L.},
% %       author={Nitsure, Nitin},
% %       author={Vistoli, Angelo},
% %        title={Fundamental algebraic geometry},
% %       series={Mathematical Surveys and Monographs},
% %    publisher={American Mathematical Society},
% %      address={Providence, RI},
% %         date={2005},
% %       volume={123},
% %         ISBN={0-8218-3541-6},
% % %       review={\MR{MR2222646}},
% % }
% 
% % \bib{fausk}{article}{
% %       author={Fausk, H.},
% %        title={Picard groups of derived categories},
% %         date={2003},
% %         ISSN={0022-4049},
% %      journal={J. Pure Appl. Algebra},
% %       volume={180},
% %       number={3},
% %        pages={251\ndash 261},
% %          url={http://dx.doi.org/10.1016/S0022-4049(02)00145-7},
% % %       review={\MR{MR1966659}},
% % }
% 
% % \bib{gabber}{article}{
% %     author={Gabber, Ofer},
% %     title={Some theorems on Azumaya algebras},
% %     conference={
% %     title={The Brauer group},
% %     address={Sem., Les Plans-sur-Bex},
% %     date={1980},
% %     },
% %     book={
% %         series={Lecture Notes in Math.},
% %         volume={844},
% %         publisher={Springer},
% %         place={Berlin},
% %     },
% %     date={1981},
% %     pages={129--209},
% % %     review={\MR{611868 (83d:13004)}},
% % }
% 
% % \bib{gaitsgory}{article}{
% %     author = {Gaitsgory, Dennis},
% %     title = {Ind-coherent sheaves},
% %  eprint = {http://arxiv.org/abs/1105:4857},
% %  journal = {ArXiv e-prints},
% %      year = {2011},
% % }
%  
% % \bib{gepner-lawson}{article}{
% %     author={Gepner, David},
% %     author={Lawson, Tyler},
% %     title={Brauer groups and Galois cohomology of commutative ring spectra},
% %     note={In preparation},
% % }
% % 
% % \bib{goerss-jardine}{book}{
% %     author={Goerss, Paul G.},
% %     author={Jardine, John F.},
% %     title={Simplicial homotopy theory},
% %     series={Progress in Mathematics},
% %     volume={174},
% %     publisher={Birkh\"auser Verlag},
% %     place={Basel},
% %     date={1999},
% %     pages={xvi+510},
% %     isbn={3-7643-6064-X},
% % %     review={\MR{1711612 (2001d:55012)}},
% % %     doi={10.1007/978-3-0348-8707-6},
% % }
% 
% % \bib{gorchinskiy-orlov}{article}{
% %     author = {Gorchinskiy, Sergey},
% %     author = {Orlov, Dmitri},
% %     title = {Geometric phantom categories},
% %     journal = {ArXiv e-prints},
% %     eprint = {http://arxiv.org/abs/1209.6183},
% %     year = {2012},
% % }
% 
% % \bib{gordon-power-street}{article}{
% %     author={Gordon, R.},
% %     author={Power, A. J.},
% %     author={Street, Ross},
% %     title={Coherence for tricategories},
% %     journal={Mem. Amer. Math. Soc.},
% %     volume={117},
% %     date={1995},
% %     number={558},
% %     pages={vi+81},
% %     issn={0065-9266},
% % %     review={\MR{1261589 (96j:18002)}},
% % }
% % 
% % \bib{ega1}{article}{
% %     author={Grothendieck, Alexander},
% %     title={\'El\'ements   de   g\'eom\'etrie   alg\'ebrique.     I.     Le    langage    des
% %     sch\'emas},
% %     journal={Inst. Hautes \'Etudes Sci. Publ. Math.},
% %     number={4},
% %     date={1960},
% %     pages={228},
% %     issn={0073-8301},
% % %     review={\MR{0217083 (36 \#177a)}},
% % }
% % 
% % \bib{ega4}{article}{
% %     author={Grothendieck, Alexander},
% %     title={\'El\'ements  de   g\'eom\'etrie   alg\'ebrique.    IV.    \'Etude   locale   des
% %     sch\'emas et des morphismes de sch\'emas IV},
% %     journal={Inst. Hautes \'Etudes Sci. Publ. Math.},
% %     number={32},
% %     date={1967},
% %     pages={361},
% %     issn={0073-8301},
% % %     review={\MR{0238860 (39 \#220)}},
% % }
% 
% % \bib{fuchs-salce}{book}{
% %     author={Fuchs, L{\'a}szl{\'o}},
% %     author={Salce, Luigi},
% %     title={Modules over non-Noetherian domains},
% %     series={Mathematical Surveys and Monographs},
% %     volume={84},
% %     publisher={American Mathematical Society},
% %     place={Providence, RI},
% %     date={2001},
% %     pages={xvi+613},
% %     isbn={0-8218-1963-1},
% % %     review={\MR{1794715 (2001i:13002)}},
% % }

% \bib{goerss-hopkins}{article}{
% author={Goerss, Paul G.},
% author={Hopkins, Michael J.},
% title={Andr\'e-Quillen (co)-homology for simplicial algebras over
% simplicial operads},
% conference={
% title={Une d\'egustation topologique [Topological morsels]:
% homotopy
% theory in the Swiss Alps},
% address={Arolla},
% date={1999},
% },
% book={
% series={Contemp.
% Math.},
% volume={265},
% publisher={Amer.
% Math.
% Soc.,
% Providence,
% RI},
% },
% date={2000},
% pages={41--85},
% % review={\MR{1803952 (2001m:18012)}},
% % doi={10.1090/conm/265/04243},
% }
% 
% % \bib{grothendieck-brauer-1}{incollection}{
% %       author={Grothendieck, Alexander},
% %        title={Le groupe de {B}rauer. {I}. {A}lg\`ebres d'{A}zumaya et
% %   interpr\'etations diverses},
% %         date={1968},
% %    booktitle={Dix {E}xpos\'es sur la {C}ohomologie des {S}ch\'emas},
% %    publisher={North-Holland},
% %      address={Amsterdam},
% %        pages={46--66},
% % %       review={\MR{MR0244269}},
% % }
% 
% % \bib{grothendieck-brauer-2}{incollection}{
% %     author={Grothendieck, Alexander},
% %     title={Le groupe de Brauer. II. Th\'eorie cohomologique},
% %     conference={
% %     title={Dix Expos\'es sur la Cohomologie des Sch\'emas},
% %     },
% %     book={
% %         publisher={North-Holland},
% %         place={Amsterdam},
% %     },
% %     date={1968},
% %     pages={67--87},
% % %     review={\MR{0244270 (39 \#5586b)}},
% % }
% % 
% % \bib{grothendieck-brauer-3}{incollection}{
% %     author={Grothendieck, Alexander},
% %     title={Le groupe de Brauer. III. Exemples et compl\'ements},
% % %     language={French},
% %     conference={
% %         title={Dix Expos\'es sur la Cohomologie des Sch\'emas},
% %     },
% %     book={
% %         publisher={North-Holland},
% %         place={Amsterdam},
% %     },
% %     date={1968},
% %     pages={88--188},
% % %     review={\MR{0244271 (39 \#5586c)}},
% % }

\bib{hesselholt-madsen}{article}{
    author={Hesselholt, Lars},
    author={Madsen, Ib},
    title={On the $K$-theory of local fields},
    journal={Ann. of Math. (2)},
    volume={158},
    date={2003},
    number={1},
    pages={1--113},
    issn={0003-486X},
%     review={\MR{1998478 (2004k:19003)}},
%     doi={10.4007/annals.2003.158.1},
}

\bib{hkr}{article}{
    author={Hochschild, G.},
    author={Kostant, Bertram},
    author={Rosenberg, Alex},
    title={Differential forms on regular affine algebras},
    journal={Trans. Amer. Math. Soc.},
    volume={102},
    date={1962},
    pages={383--408},
    issn={0002-9947},
%     review={\MR{0142598}},
%     doi={10.2307/1993614},
}
 
% % \bib{hopkins-mahowald-sadofsky}{incollection}{
% %       author={Hopkins, Michael~J.},
% %       author={Mahowald, Mark},
% %       author={Sadofsky, Hal},
% %        title={Constructions of elements in {P}icard groups},
% %         date={1994},
% %    booktitle={Topology   and    representation    theory    ({E}vanston,    {IL},    1992)},
% %       series={Contemp. Math.},
% %       volume={158},
% %    publisher={Amer. Math. Soc.},
% %      address={Providence, RI},
% %        pages={89\ndash 126},
% % %       review={\MR{MR1263713}},
% % }
% 
% \bib{hovey-palmieri-strickland}{article}{
%     author={Hovey, Mark},
%     author={Palmieri, John H.},
%     author={Strickland, Neil P.},
%     title={Axiomatic stable homotopy theory},
%     journal={Mem. Amer. Math. Soc.},
%     volume={128},
%     date={1997},
%     number={610},
%     pages={x+114},
%     issn={0065-9266},
% %     review={\MR{1388895 (98a:55017)}},
% }
% 
% % \bib{hovey-shipley-smith}{article}{
% %     author={Hovey, Mark},
% %     author={Shipley, Brooke},
% %     author={Smith, Jeff},
% %     title={Symmetric spectra},
% %     journal={J. Amer. Math. Soc.},
% %     volume={13},
% %     date={2000},
% %     number={1},
% %     pages={149--208},
% %     issn={0894-0347},
% % %     review={\MR{1695653 (2000h:55016)}},
% % %     doi={10.1090/S0894-0347-99-00320-3},
% % }

% \bib{hovey-strickland}{article}{
%     author={Hovey, Mark},
%     author={Strickland, Neil P.},
%     title={Morava $K$-theories and localisation},
%     journal={Mem. Amer. Math. Soc.},
%     volume={139},
%     date={1999},
%     number={666},
%     pages={viii+100},
%     issn={0065-9266},
% %     review={\MR{1601906 (99b:55017)}},
% %     doi={10.1090/memo/0666},
% }

% 
% % \bib{huybrechts}{book}{
% %     author={Huybrechts, D.},
% %     title={Fourier-Mukai transforms in algebraic geometry},
% %     series={Oxford Mathematical Monographs},
% %     publisher={The Clarendon Press Oxford University Press},
% %     place={Oxford},
% %     date={2006},
% %     pages={viii+307},
% %     isbn={978-0-19-929686-6},
% %     isbn={0-19-929686-3},
% % %     review={\MR{2244106 (2007f:14013)}},
% % %     doi={10.1093/acprof:oso/9780199296866.001.0001},
% % }
% 
% % \bib{johnson-azumaya}{article}{
% %     author = {Johnson, Niles},
% %     title = {Azumaya Objects in Triangulated Bicategories},
% %     journal = {ArXiv e-prints},
% %     eprint = {http://arxiv.org/abs/1005.4878},
% %     year = {2010},
% % }

\bib{jones}{article}{
    author={Jones, John D. S.},
    title={Cyclic homology and equivariant homology},
    journal={Invent. Math.},
    volume={87},
    date={1987},
    number={2},
    pages={403--423},
    issn={0020-9910},
%     review={\MR{870737 (88f:18016)}},
%     doi={10.1007/BF01389424},
}
% 
% % \bib{knus-ojanguren}{book}{
% %       author={Knus, Max-Albert},
% %       author={Ojanguren, Manuel},
% %        title={Th\'eorie de la descente et alg\`ebres d'{A}zumaya},
% %       series={Lecture Notes in Mathematics, Vol. 389},
% %    publisher={Springer-Verlag},
% %      address={Berlin},
% %         date={1974},
% % %       review={\MR{MR0417149}},
% % }
% 
% % \bib{kapranov}{article}{
% %     author={Kapranov, M.},
% %     title={Noncommutative geometry based on commutator expansions},
% %     journal={J. Reine Angew. Math.},
% %     volume={505},
% %     date={1998},
% %     pages={73--118},
% %     issn={0075-4102},
% % %     review={\MR{1662244 (2000b:14003)}},
% % %     doi={10.1515/crll.1998.122},
% % }
% 
% \bib{keller-smashing}{article}{
%     author={Keller, Bernhard},
%     title={A remark on the generalized smashing conjecture},
%     journal={Manuscripta Math.},
%     volume={84},
%     date={1994},
%     number={2},
%     pages={193--198},
%     issn={0025-2611},
% %     review={\MR{1285956 (95h:18014)}},
% %     doi={10.1007/BF02567453},
% }
% 
% % \bib{krause}{article}{
% %     author={Krause, Henning},
% %     title={The stable derived category of a Noetherian scheme},
% %     journal={Compos. Math.},
% %     volume={141},
% %     date={2005},
% %     number={5},
% %     pages={1128--1162},
% %     issn={0010-437X},
% % %     review={\MR{2157133 (2006e:18019)}},
% % %     doi={10.1112/S0010437X05001375},
% % }
% 
% \bib{krause-localization}{article}{
%     author={Krause, Henning},
%     title={Localization theory for triangulated categories},
%     conference={
%         title={Triangulated categories},
%     },
%     book={
%         series={London Math. Soc. Lecture Note Ser.},
%         volume={375},
%         publisher={Cambridge Univ. Press},
%         place={Cambridge},
%     },
%     date={2010},
%     pages={161--235},
% %     review={\MR{2681709 (2012e:18026)}},
% %     doi={10.1017/CBO9781139107075.005},
% }
% 
% 
% \bib{krause-stovicek}{article}{
%     author={Krause, Henning},
%     author={\v{S}\v{t}ov{\'{\i}}{\v{c}}ek, Jan},
%     title={The telescope conjecture for hereditary rings via Ext-orthogonal
%     pairs},
%     journal={Adv. Math.},
%     volume={225},
%     date={2010},
%     number={5},
%     pages={2341--2364},
%     issn={0001-8708},
% %     review={\MR{2680168 (2011j:16013)}},
% %     doi={10.1016/j.aim.2010.04.027},
% }
% 
% % \bib{kresch-vistoli}{article}{
% %     author={Kresch, Andrew},
% %     author={Vistoli, Angelo},
% %     title={On coverings of Deligne-Mumford stacks and surjectivity of the Brauer map},
% %     journal={Bull. London Math. Soc.},
% %     volume={36},
% %     date={2004},
% %     number={2},
% %     pages={188--192},
% %     issn={0024-6093},
% % %     review={\MR{2026412 (2004j:14003)}},
% % %     doi={10.1112/S0024609303002728},
% % }
% 
% % \bib{laumon-moret-bailly}{book}{
% %     author={Laumon, G{\'e}rard},
% %     author={Moret-Bailly, Laurent},
% %     title={Champs alg\'ebriques},
% %     series={Ergebnisse der Mathematik und ihrer Grenzgebiete. 3. Folge. A
% %         Series of Modern Surveys in Mathematics},
% %     volume={39},
% %     publisher={Springer-Verlag},
% %     place={Berlin},
% %     date={2000},
% %     pages={xii+208},
% %     isbn={3-540-65761-4},
% % %     review={\MR{1771927 (2001f:14006)}},
% % }
% 
% % \bib{lazarev}{article}{
% %       author={Lazarev, A.},
% %        title={Homotopy  theory  of   {$A_\infty$}   ring   spectra   and   applications   to
% %   {$MU$}-modules},
% %         date={2001},
% %         ISSN={0920-3036},
% %      journal={$K$-Theory},
% %       volume={24},
% %       number={3},
% %        pages={243\ndash 281},
% %          url={http://dx.doi.org/10.1023/A:1013394125552},
% % %       review={\MR{MR1876800}},
% % }
% % 
% % \bib{lieblich-thesis}{thesis}{
% %     author={Lieblich, Max},
% %     title={Moduli of twisted sheaves and generalized Azumaya algebras},
% %     note={Thesis (Ph.D.)--Massachusetts Institute of Technology},
% % %     publisher={ProQuest LLC, Ann Arbor, MI},
% %     date={2004},
% % %     pages={(no paging)},
% % %     review={\MR{2717173}},
% % }
% 
% \bib{lieblich-moduli}{article}{
%     author={Lieblich, Max},
%     title={Moduli of twisted sheaves},
%     journal={Duke Math. J.},
%     volume={138},
%     date={2007},
%     number={1},
%     pages={23--118},
%     issn={0012-7094},
% %     review={\MR{2309155 (2008d:14018)}},
% %     doi={10.1215/S0012-7094-07-13812-2},
% }
% 
% % \bib{lieblich}{article}{
% %     author={Lieblich, Max},
% %     title={Twisted sheaves and the period-index problem},
% %     journal={Compos. Math.},
% %     volume={144},
% %     date={2008},
% %     number={1},
% %     pages={1--31},
% %     issn={0010-437X},
% % %     review={\MR{2388554 (2009b:14033)}},
% % %     doi={10.1112/S0010437X07003144},
% % }
\bib{loday}{book}{
    author={Loday, Jean-Louis},
    title={Cyclic homology},
    series={Grundlehren der Mathematischen Wissenschaften},
    volume={301},
    edition={2},
    publisher={Springer-Verlag, Berlin},
    date={1998},
    pages={xx+513},
    isbn={3-540-63074-0},
}
% 
% % \bib{lunts}{article}{
% %     author={Lunts, Valery A.},
% %     title={Categorical resolution of singularities},
% %     journal={J. Algebra},
% %     volume={323},
% %     date={2010},
% %     number={10},
% %     pages={2977--3003},
% %     issn={0021-8693},
% % %     review={\MR{2609187 (2011d:18019)}},
% % %     doi={10.1016/j.jalgebra.2009.12.023},
% % }
% 
% \bib{htt}{book}{
%       author={Lurie, Jacob},
%        title={Higher topos theory},
%       series={Annals of Mathematics Studies},
%    publisher={Princeton University Press},
%      address={Princeton, NJ},
%         date={2009},
%       volume={170},
%         ISBN={978-0-691-14049-0; 0-691-14049-9},
% %       review={\MR{MR2522659}},
% }
% 
% % \bib{dag0}{thesis}{
% %       author={Lurie, Jacob},
% %       title={Derived algebraic geometry},
% %         date={2004},
% %         note={PhD thesis. Available at
% %             \texttt{\href{http://hdl.handle.net/1721.1/30144}{http://hdl.handle.net/1721.1/30144}}},
% % }
% 
% % \bib{dag7}{unpublished}{
% %       author={Lurie, Jacob},
% %       title={Derived algebraic geometry VII: spectral schemes},
% %         date={2011},
% %   note={\texttt{\href{http://www.math.harvard.edu/~lurie/}{http://www.math.harvard.edu/\textasciitilde lurie/}}},
% % }
% 
% \bib{dag11}{article}{
%     author={Lurie, Jacob},
%     title={Derived algebraic geometry XI: descent theorems},
%     date={2011},
%     eprint={http://www.math.harvard.edu/~lurie/},
% }

%  \bib{dag12}{article}{
%      author={Lurie, Jacob},
%      title={Derived algebraic geometry XII: proper morphisms, completions, and the Grothendieck existence theorem },
%      date={2011},
%      eprint={http://www.math.harvard.edu/~lurie/},
%      note={Version dated 8 Novermber 2011},
%  }
% 
% % \bib{dag14}{unpublished}{
% %       author={Lurie, Jacob},
% %       title={Derived algebraic geometry XIV: representability theorems},
% %         date={2012},
% %   note={\texttt{\href{http://www.math.harvard.edu/~lurie/}{http://www.math.harvard.edu/\textasciitilde lurie/}}},
% % }
% 
\bib{ha}{article}{
    author={Lurie, Jacob},
    title={Higher algebra},
    date={2016},
    eprint={http://www.math.harvard.edu/~lurie/},
    note={Version dated 10 March 2016},
}
% 

% \bib{mahowald-ravenel-shick}{article}{
%    author={Mahowald, Mark},
%    author={Ravenel, Douglas},
%    author={Shick, Paul},
%    title={The triple loop space approach to the telescope conjecture},
%    conference={
%       title={Homotopy methods in algebraic topology},
%       address={Boulder, CO},
%       date={1999},
%    },
%    book={
%       series={Contemp. Math.},
%       volume={271},
%       publisher={Amer. Math. Soc., Providence, RI},
%    },
%    date={2001},
%    pages={217--284},
% %    review={\MR{1831355 (2002g:55014)}},
% %    doi={10.1090/conm/271/04358},
% }


\bib{mccarthy-minasian}{article}{
    author={McCarthy, Randy},
    author={Minasian, Vahagn},
    title={H{KR} theorem for smooth {$S$}-algebras},
    journal={J. Pure Appl. Algebra},
    volume={185},
    date={2003},
    number={1-3},
    pages={239--258},
    issn={0022-4049},
}

\bib{mcclure-schwanzl-vogt}{article}{
    author={McClure, J.},
    author={Schw{\"a}nzl, R.},
    author={Vogt, R.},
    title={$THH(R)\cong R\otimes S^1$ for $E_\infty$ ring spectra},
    journal={J. Pure Appl. Algebra},
    volume={121},
    date={1997},
    number={2},
    pages={137--159},
    issn={0022-4049},
%     review={\MR{1473888 (98k:55010)}},
%     doi={10.1016/S0022-4049(97)00118-7},
}

% \bib{mcclure-staffeldt}{article}{
%     author={McClure, J. E.},
%     author={Staffeldt, R. E.},
%     title={On the topological {H}ochschild homology of {$b{\rm u}$}. {I}},
%     journal={Amer. J. Math.},
%     volume={115},
%     date={1993},
%     number={1},
%     pages={1--45},
%     issn={0002-9327},
% }
% \bib{mcclure-staffeldt-chromatic}{article}{
%     author={McClure, J. E.},
%     author={Staffeldt, R. E.},
%     title={The chromatic convergence theorem and a tower in algebraic
%     $K$-theory},
%     journal={Proc. Amer. Math. Soc.},
%     volume={118},
%     date={1993},
%     number={3},
%     pages={1005--1012},
%     issn={0002-9939},
% %     review={\MR{1164148 (93j:55012)}},
% %     doi={10.2307/2160154},
% }

%
% % \bib{matsumura}{book}{
% %     author={Matsumura, Hideyuki},
% %     title={Commutative ring theory},
% %     series={Cambridge Studies in Advanced Mathematics},
% %     volume={8},
% %     edition={2},
% %     note={Translated from the Japanese by M. Reid},
% %     publisher={Cambridge University Press},
% %     place={Cambridge},
% %     date={1989},
% %     pages={xiv+320},
% %     isbn={0-521-36764-6},
% % %     review={\MR{1011461 (90i:13001)}},
% % }
% 
% \bib{neeman-chromatic}{article}{
%     author={Neeman, Amnon},
%     title={The chromatic tower for $D(R)$},
%     note={With an appendix by Marcel B\"okstedt},
%     journal={Topology},
%     volume={31},
%     date={1992},
%     number={3},
%     pages={519--532},
%     issn={0040-9383},
% %     review={\MR{1174255 (93h:18018)}},
% %     doi={10.1016/0040-9383(92)90047-L},
% }
% 
% \bib{neeman-1992}{article}{
%     author={Neeman, Amnon},
%     title={The connection between the $K$-theory localization theorem of
%     Thomason, Trobaugh and Yao and the smashing subcategories of Bousfield
%     and Ravenel},
%     journal={Ann. Sci. \'Ecole Norm. Sup. (4)},
%     volume={25},
%     date={1992},
%     number={5},
%     pages={547--566},
%     issn={0012-9593},
% %     review={\MR{1191736 (93k:18015)}},
% }
% 
% % \bib{neeman-1996}{article}{
% %     author={Neeman, Amnon},
% %     title={The Grothendieck duality theorem via Bousfield's techniques and Brown representability},
% %     journal={J. Amer. Math. Soc.},
% %     volume={9},
% %     date={1996},
% %     number={1},
% %     pages={205--236},
% %     issn={0894-0347},
% % %     review={\MR{1308405 (96c:18006)}},
% % %     doi={10.1090/S0894-0347-96-00174-9},
% % }
\bib{neeman-ranicki}{article}{
    author={Neeman, Amnon},
    author={Ranicki, Andrew},
    title={Noncommutative localisation in algebraic $K$-theory. I},
    journal={Geom. Topol.},
    volume={8},
    date={2004},
    pages={1385--1425},
    issn={1465-3060},
%     review={\MR{2119300 (2005k:19006)}},
%     doi={10.2140/gt.2004.8.1385},
}
% 
% % \bib{orlov}{article}{
% %     author={Orlov, D. O.},
% %     title={Projective bundles, monoidal transformations, and derived categories of coherent sheaves},
% %     journal={Izv. Ross. Akad. Nauk Ser. Mat.},
% %     volume={56},
% %     date={1992},
% %     number={4},
% %     pages={852--862},
% %     issn={0373-2436},
% %     translation={
% %         journal={Russian Acad. Sci. Izv. Math.},
% %         volume={41},
% %         date={1993},
% %         number={1},
% %         pages={133--141},
% %         issn={1064-5632},
% %     },
% % %     review={\MR{1208153 (94e:14024)}},
% % %     doi={10.1070/IM1993v041n01ABEH002182},
% % }
% 
% % \bib{preygel}{article}{
% %     author = {Preygel, Anatoly},
% %     title = {Thom-Sebastiani and duality for matrix factorizations},
% %  eprint = {http://arxiv.org/abs/1101:5834},
% %  journal = {ArXiv e-prints},
% %      year = {2011},
% % }
% % 
\bib{quillen}{article}{
    author={Quillen, Daniel},
    title={Higher algebraic $K$-theory. I},
    conference={
        title={Algebraic $K$-theory, I: Higher $K$-theories (Proc. Conf.,
        Battelle Memorial Inst., Seattle, Wash., 1972)},
    },
    book={
        publisher={Springer},
        place={Berlin},
    },
    date={1973},
    pages={85--147. Lecture Notes in Math., Vol.  341},
%     review={\MR{0338129 (49 \#2895)}},
}
% 
% \bib{ravenel-localization}{article}{
%     author={Ravenel, Douglas C.},
%     title={Localization with respect to certain periodic homology theories},
%     journal={Amer. J. Math.},
%     volume={106},
%     date={1984},
%     number={2},
%     pages={351--414},
%     issn={0002-9327},
% %     review={\MR{737778 (85k:55009)}},
% %     doi={10.2307/2374308},
% }

% \bib{ravenel-nilpotence}{book}{
%     author={Ravenel, Douglas C.},
%     title={Nilpotence and periodicity in stable homotopy theory},
%     series={Annals of Mathematics Studies},
%     volume={128},
%     note={Appendix C by Jeff Smith},
%     publisher={Princeton University Press, Princeton, NJ},
%     date={1992},
%     pages={xiv+209},
%     isbn={0-691-02572-X},
% %     review={\MR{1192553 (94b:55015)}},
% }

% \bib{rezk-notes}{article}{
%     author={Rezk, Charles},
%     title={Notes on the Hopkins-Miller theorem},
%     conference={
%     title={Homotopy theory via algebraic geometry and group
%     representations },
%     address={Evanston, IL},
%     date={1997},
%     },
%     book={
%     series={Contemp. Math.},
%     volume={220},
%     publisher={Amer.
%     Math. Soc.,
%     Providence, RI},
%     },
%     date={1998},
%     pages={313--366},
% %     review={\MR{1642902
% %     (2000i:55023)}},
% %     doi={10.1090/conm/220/03107},
% }

\bib{rognes-2primary}{article}{
    author={Rognes, John},
    title={Two-primary algebraic $K$-theory of pointed spaces},
    journal={Topology},
    volume={41},
    date={2002},
    number={5},
    pages={873--926},
    issn={0040-9383},
%     review={\MR{1923990 (2003m:19002)}},
%     doi={10.1016/S0040-9383(01)00005-2},
}

\bib{rognes-whitehead}{article}{
    author={Rognes, John},
    title={The smooth Whitehead spectrum of a point at odd regular primes},
    journal={Geom. Topol.},
    volume={7},
    date={2003},
    pages={155--184},
    issn={1465-3060},
%     review={\MR{1988283 (2004f:19004)}},
%     doi={10.2140/gt.2003.7.155},
}

% 
% % \bib{rognes-book}{article}{
% % author={Rognes, John},
% % title={Galois extensions of structured ring spectra. Stably dualizable
% % groups},
% % journal={Mem. Amer. Math. Soc.},
% % volume={192},
% % date={2008},
% % number={898},
% % pages={viii+137},
% % issn={0065-9266},
% % % review={\MR{2387923 (2009c:55007)}},
% % }
% 
% % \bib{sga6}{book}{
% %     label={SGA6},
% %     title={Th\'eorie des intersections et th\'eor\`eme de Riemann-Roch},
% %     series={Lecture Notes in Mathematics, Vol. 225},
% %     note={S\'eminaire de G\'eom\'etrie Alg\'ebrique du Bois-Marie 1966--1967
% %     (SGA 6);
% %     Dirig\'e   par   P.    Berthelot,   A.    Grothendieck   et   L.     Illusie.     Avec    la
% %     collaboration  de  D.    Ferrand,   J.    P.    Jouanolou,   O.    Jussila,   S.    Kleiman,
% %     M.
% %     Raynaud et J. P. Serre},
% %     publisher={Springer-Verlag},
% %     place={Berlin},
% %     date={1971},
% %     pages={xii+700},
% % % review={\MR{0354655 (50 \#7133)}},
% % }
%  
% % \bib{schwede}{article}{
% %     author={Schwede, Stefan},
% %     title={The stable homotopy category is rigid},
% %     journal={Ann. of Math. (2)},
% %     volume={166},
% %     date={2007},
% %     number={3},
% %     pages={837--863},
% %     issn={0003-486X},
% % %     review={\MR{2373374 (2009g:55009)}},
% % %     doi={10.4007/annals.2007.166.837},
% % }
% % 
% % \bib{ss-uniqueness}{article}{
% %     author={Schwede, Stefan},
% %     author={Shipley, Brooke},
% %     title={A uniqueness theorem for stable homotopy theory},
% %     journal={Math. Z.},
% %     volume={239},
% %     date={2002},
% %     number={4},
% %     pages={803--828},
% %     issn={0025-5874},
% % %     review={\MR{1902062 (2003f:55027)}},
% % %     doi={10.1007/s002090100347},
% % }
% % 

\bib{schlichtkrull-units}{article}{
   author={Schlichtkrull, Christian},
   title={Units of ring spectra and their traces in algebraic $K$-theory},
   journal={Geom. Topol.},
   volume={8},
   date={2004},
   pages={645--673},
   issn={1465-3060},
%    review={\MR{2057776 (2005m:19003)}},
%    doi={10.2140/gt.2004.8.645},
}

\bib{schwede-shipley}{article}{
    author={Schwede, Stefan},
    author={Shipley, Brooke},
    title={Stable model categories are categories of modules},
    journal={Topology},
    volume={42},
    date={2003},
    number={1},
    pages={103--153},
    issn={0040-9383},
    % review={\MR{1928647 (2003g:55034)}},
    % doi={10.1016/S0040-9383(02)00006-X},
}

\bib{shipley-hochschild}{article}{
    author={Shipley, Brooke},
    title={Symmetric spectra and topological Hochschild homology},
    journal={$K$-Theory},
    volume={19},
    date={2000},
    number={2},
    pages={155--183},
    issn={0920-3036},
%     review={\MR{1740756 (2001h:55010)}},
%     doi={10.1023/A:1007892801533},
}
\bib{shipley-hz}{article}{
    author={Shipley, Brooke},
    title={$H\mathbb{Z}$-algebra spectra are differential graded algebras},
    journal={Amer. J. Math.},
    volume={129},
    date={2007},
    number={2},
    pages={351--379},
    issn={0002-9327},
%     review={\MR{2306038 (2008b:55015)}},
%     doi={10.1353/ajm.2007.0014},
}
% % 
% % \bib{simpson}{article}{
% %     author = {Simpson, Carlos},
% %     title = {Algebraic (geometric) {$n$}-stacks},
% %  eprint = {http://arxiv.org/absalg-geom/9609014},
% %  journal = {ArXiv e-prints},
% %      year = {1996},
% % }
% 
% % \bib{sosna-scalar}{article}{
% %     author = {Sosna, Pawel},
% %     title = {Scalar extensions of triangulated categories},
% %  eprint = {http://arxiv.org/abs/1109.6515},
% %  journal = {ArXiv e-prints},
% %      year = {2011},
% % }
% 
% 
% \bib{stevenson-support}{article}{
%     author = {Stevenson, Greg},
%     title  =  {Support theory via actions of tensor triangulated categories},
%     journal = {J. Reine Agnew. Math.},
%     volume = {681},
%     year = {2013},
%     pages = {219--254},
% }
% 
% \bib{stevenson-singularity}{article}{
%     author = {Stevenson, Greg},
%     title  =  {Subcategories of singularity categories via tensor actions},
%     journal = {to appear in Compos. Math.},
%     eprint = {http://arxiv.org/abs/1105.4698},
% }
% 
% \bib{stevenson-flat}{article}{
%     author = {Stevenson, Greg},
%     title  =  {Derived categories of absolutely flat rings},
%     journal = {ArXiv e-prints},
%     eprint = {http://arxiv.org/abs/1210.0399},
%     year = {2012},
% }
\bib{strickland}{article}{
    author={Strickland, N. P.},
    title={Products on ${\rm MU}$-modules},
    journal={Trans. Amer. Math. Soc.},
    volume={351},
    date={1999},
    number={7},
    pages={2569--2606},
    issn={0002-9947},
%     review={\MR{1641115 (2000b:55003)}},
%     doi={10.1090/S0002-9947-99-02436-8},
}
% 
% % \bib{swan}{article}{
% %     author={Swan, Richard G.},
% %     title={Hochschild cohomology of quasiprojective schemes},
% %     journal={J. Pure Appl. Algebra},
% %     volume={110},
% %     date={1996},
% %     number={1},
% %     pages={57--80},
% %     issn={0022-4049},
% % %     review={\MR{1390671 (97j:19003)}},
% % %     doi={10.1016/0022-4049(95)00091-7},
% % }
%  
% % \bib{szymik}{article}{
% %     author = {{Szymik}, M.},
% %     title  =  {Brauer  spaces  for  commutative  rings   and   structured   ring   spectra},
% %     journal = {ArXiv e-prints},
% %     eprint = {http://arxiv.org/abs/1110.2956},
% %     year = {2011},
% % }
%  
% \bib{thomason-triangulated}{article}{
%     author={Thomason, R. W.},
%     title={The classification of triangulated subcategories},
%     journal={Compos. Math.},
%     volume={105},
%     date={1997},
%     number={1},
%     pages={1--27},
%     issn={0010-437X},
% %     review={\MR{1436741 (98b:18017)}},
% %     doi={10.1023/A:1017932514274},
% }
%  
\bib{thomason-trobaugh}{article}{
    author={Thomason, R. W.},
    author={Trobaugh, Thomas},
    title={Higher  algebraic   $K$-theory   of   schemes   and   of   derived   categories},
    conference={
    title={The Grothendieck Festschrift, Vol.\ III},
    },
    book={
        series={Progr. Math.},
        volume={88},
        publisher={Birkh\"auser Boston},
        place={Boston, MA},
    },
    date={1990},
    pages={247--435},
    % review={\MR{1106918 (92f:19001)}},
    % doi={10.1007/978-0-8176-4576-2_10},
}
% 
% % \bib{toen-morita}{article}{
% %     author={To{\"e}n, Bertrand},
% %     title={The homotopy theory of $dg$-categories and derived Morita theory},
% %     journal={Invent. Math.},
% %     volume={167},
% %     date={2007},
% %     number={3},
% %     pages={615--667},
% %     issn={0020-9910},
% % %     review={\MR{2276263 (2008a:18006)}},
% % %     doi={10.1007/s00222-006-0025-y},
% % }
% 
% % \bib{toen-descente}{article}{
% %     author = {To{\"e}n, Bertrand},
% %     title = {Descente fid{\`e}lement plate pour les {n}-champs d'Artin},
% %     journal = {Compositio Mathematica},
% %     volume = {147},
% %     number = {5},
% %     year = {2011},
% %     pages = {1382-1412},
% % }
%    
% 
% \bib{toen-derived}{article}{
%     author = {To{\"e}n, Bertrand},
%     title = {Derived Azumaya algebras and generators for twisted derived categories},
%     journal = {Invent. Math.},
%     year = {2012},
%     volume = {189},
%     number = {3},
%     pages = {581--652},
% }
% 
% % \bib{toen-derived}{article}{
% % author = {{Toen}, B.},
% % title  =  {Derived  Azumaya  algebras  and  generators  for  twisted  derived   categories},
% % journal = {ArXiv e-prints},
% % eprint = {http://arxiv.org/abs/1002.2599},
% % year = {2010},
% % }
% 
% % \bib{toen-vaquie}{article}{
% %     author={To{\"e}n, Bertrand},
% %     author={Vaqui{\'e}, Michel},
% %     title={Moduli of objects in dg-categories},
% %     journal={Ann. Sci. \'Ecole Norm. Sup. (4)},
% %     volume={40},
% %     date={2007},
% %     number={3},
% %     pages={387--444},
% %     issn={0012-9593},
% % %     review={\MR{2493386 (2009m:18015)}},
% % %     doi={10.1016/j.ansens.2007.05.001},
% % }
% 
% % \bib{hag2}{article}{
% %     author={To{\"e}n, Bertrand},
% %     author={Vezzosi, Gabriele},
% %     title={Homotopical algebraic geometry. II. Geometric stacks and applications},
% %     journal={Mem. Amer. Math. Soc.},
% %     volume={193},
% %     date={2008},
% %     number={902},
% %     pages={x+224},
% %     issn={0065-9266},
% % %     review={\MR{2394633 (2009h:14004)}},
% % }
% 
% % \bib{vanoystaeyen-zhang}{article}{
% %     author={Van Oystaeyen, Fred},
% %     author={Zhang, Yinhuo},
% %     title={The Brauer group of a braided monoidal category},
% %     journal={J. Algebra},
% %     volume={202},
% %     date={1998},
% %     number={1},
% %     pages={96--128},
% %     issn={0021-8693},
% % %     review={\MR{1614178 (99c:18006)}},
% % %     doi={10.1006/jabr.1997.7295},
% % }
% 
% % \bib{vandenbergh}{article}{
% %     author={Van den Bergh, Michel},
% %     title={Three-dimensional flops and noncommutative rings},
% %     journal={Duke Math. J.},
% %     volume={122},
% %     date={2004},
% %     number={3},
% %     pages={423--455},
% %     issn={0012-7094},
% % %     review={\MR{2057015 (2005e:14023)}},
% % %     doi={10.1215/S0012-7094-04-12231-6},
% % }
% 
% % \bib{vitale-brauer}{article}{
% %     author={Vitale, Enrico M.},
% %     title={The Brauer and Brauer-Taylor groups of a symmetric monoidal category},
% %     journal={Cahiers Topologie G\'eom. Diff\'erentielle Cat\'eg.},
% %     volume={37},
% %     date={1996},
% %     number={2},
% %     pages={91--122},
% %     issn={0008-0004},
% % %     review={\MR{1394505 (97i:18011)}},
% % }

% \bib{waldhausen-1}{article}{
%     author={Waldhausen, Friedhelm},
%     title={Algebraic $K$-theory of topological spaces. I},
%     conference={
%     title={Algebraic and geometric topology},
%     address={Proc. Sympos. Pure Math., Stanford Univ., Stanford,
%     Calif.},
%     date={1976},
%     },
%     book={
%     series={Proc. Sympos. Pure Math., XXXII},
%     publisher={Amer. Math. Soc.,
%     Providence, R.I.},
%     },
%     date={1978},
%     pages={35--60},
% %     review={\MR{520492 (81i:18014a)}},
% }

\bib{waldhausen-2}{article}{
    author={Waldhausen, Friedhelm},
    title={Algebraic $K$-theory of topological spaces. II},
    conference={
    title={Algebraic topology, Aarhus 1978 (Proc. Sympos., Univ. Aarhus,
    Aarhus, 1978)},
    },
    book={
    series={Lecture Notes in Math.},
    volume={763},
    publisher={Springer, Berlin},
    },
    date={1979},
    pages={356--394},
%     review={\MR{561230 (81i:18014b)}},
}

% \bib{waldhausen-chromatic}{article}{
%     author={Waldhausen, Friedhelm},
%     title={Algebraic $K$-theory of spaces, localization, and the chromatic
%     filtration of stable homotopy},
%     conference={
%     title={Algebraic topology, Aarhus 1982},
%     address={Aarhus},
%     date={1982},
%     },
%     book={
%     series={Lecture Notes in Math.},
%     volume={1051},
%     publisher={Springer,
%     Berlin},
%     },
%     date={1984},
%     pages={173--195},
% %     review={\MR{764579 (86c:57016)}},
% %     doi={10.1007/BFb0075567},
% }

% \bib{waldhausen}{article}{
%     author={Waldhausen, Friedhelm},
%     title={Algebraic $K$-theory of spaces},
%     conference={
%     title={Algebraic and geometric topology},
%     address={New Brunswick, N.J.},
%     date={1983},
%     },
%     book={
%     series={Lecture Notes in Math.},
%     volume={1126},
%     publisher={Springer, Berlin},
%     },
%     date={1985},
%     pages={318--419},
% %     review={\MR{802796 (86m:18011)}},
% %     doi={10.1007/BFb0074449},
% }

% % \bib{wang}{article}{
% %     author={Wang, Shianghaw},
% %     title={On the commutator group of a simple algebra},
% %     journal={Amer. J. Math.},
% %     volume={72},
% %     date={1950},
% %     pages={323--334},
% %     issn={0002-9327},
% % %     review={\MR{0034380 (11,577d)}},
% % }

\bib{weibel-homological}{book}{
    author={Weibel, Charles A.},
    title={An introduction to homological algebra},
    series={Cambridge Studies in Advanced Mathematics},
    volume={38},
    publisher={Cambridge University Press, Cambridge},
    date={1994},
    pages={xiv+450},
    isbn={0-521-43500-5},
    isbn={0-521-55987-1},
%     review={\MR{1269324 (95f:18001)}},
%     doi={10.1017/CBO9781139644136},
}

% \bib{weibel-kchapter}{article}{
%     author={Weibel, Charles},
%     title={Algebraic $K$-theory of rings of integers in local and global
%     fields},
%     conference={
%     title={Handbook of $K$-theory. Vol. 1, 2},
%     },
%     book={
%     publisher={Springer, Berlin},
%     },
%     date={2005},
%     pages={139--190},
% %     review={\MR{2181823 (2006g:11232)}},
% %     doi={10.1007/3-540-27855-9_5},
% }
\bib{weibel}{book}{
    author={Weibel, Charles A.},
    title={The $K$-book},
    series={Graduate Studies in Mathematics},
    volume={145},
    publisher={American Mathematical Society, Providence, RI},
    date={2013},
    pages={xii+618},
%     isbn={978-0-8218-9132-2},
%     review={\MR{3076731}},
}

\end{biblist}
\end{bibdiv}
\end{document}